\documentclass[11pt,letterpaper]{amsart}
\usepackage{mathrsfs}

\usepackage{color}
\usepackage{bm}
\usepackage{amsfonts,amssymb,verbatim}
\usepackage{dsfont}
\usepackage{amscd}
\usepackage{extarrows}
\usepackage{amsmath}
\usepackage{mathrsfs}
\usepackage{enumerate}
\usepackage{amscd}
\usepackage[all]{xy}
\usepackage{hyperref}
\usepackage[obeyspaces]{xurl}
\numberwithin{equation}{section}
\theoremstyle{plain} 
\newtheorem{theorem}{\indent\bf Theorem}[section]

\theoremstyle{definition} 

\newtheorem{remark}[theorem]{\bf Remark}

\newtheorem{thm}{Theorem}[section]
\newtheorem{cor}[thm]{Corollary}
\newtheorem{lem}[thm]{Lemma}

\theoremstyle{definition}
\newtheorem{defn}{Definition}[section]

\theoremstyle{remark}

	\newcommand{\mc}{\mathbb{C}}
	\newcommand{\mr}{\mathbb{R}}
	
	\newcommand{\ra}{\rightarrow}
	\newcommand{\bi}{\begin{itemize}}
		\newcommand{\ei}{\end{itemize}}

	\newcommand{\mo}{\mathcal{O}}
	
	\newcommand{\rw}{\rightarrow}
	\newcommand{\rwo}{\mapsto}

    \newcommand{\bing}{\mathop{\cup}}

	\allowdisplaybreaks[4]
	
\begin{document}
		
		\title[curvature strict positivity of direct image bundles ]
		{Curvature formulas and curvature strict positivity of direct image bundles}

		\author[X. Qin]{Xiangsen Qin}
		\address{Xiangsen Qin: \ School of Mathematical Sciences, University of Chinese Academy of Sciences\\ Beijing 100049, P. R. China}
		\email{qinxiangsen19@mails.ucas.ac.cn}

		\begin{abstract}
       In this paper, we consider the curvature strict positivity of direct image bundles (vector bundles) associated to a strictly pseudoconvex family of bounded domains.
       The main result is that the curvature of the direct image bundle associated to a strictly pseudoconvex family of bounded domains is strictly positive in the sense of Nakano
       even if the curvature of the original vector bundle is just Nakano positive. Based on our (I and my coauthors) previous results, this result further demonstrates 
       that strictly pseudoconvex domains and pseudoconvex domains have very different geometric properties. To consider the curvature strict positivity, we will first construct
       a curvature formula for a direct image bundle, then the curvature strict positivity will be a simple consequence of it.  As applications to convex analysis, we get a corresponding version of Pr\'ekopa's Theorem, i.e., we get the curvature strict positivity of a strictly convex family of bounded domains.
       In the last, we give a flatness criterion for the direct image bundles.
		\end{abstract}

		\maketitle
		
		\tableofcontents
		\section{Introduction}
		Let $U\subset \mc_t^n$ and $D\subset \mc_z^m$ be bounded pseudoconvex domains, and let $\varphi$ be a smooth plurisubharmonic function defined on some open neighborhood of 
        the closure of $\Omega:=U\times D$ in $\mc^n\times\mc^m$.
		For $t\in U$, define the Hilbert space
		$$E_t:=\{f\in \mathcal O(D)|\ \|f\|^2_t:=\int_{D}|f(z)|^2e^{-\varphi(t,z)}d\lambda(z)<\infty\},$$
		where $\mathcal O(D)$ is the space of holomorphic functions on $D$,  and $d\lambda$ is the Lebesgue measure on $\mc^m$.
		When $t$ varies in $U$, $E_t$ is invariant as a vector space, but the underlying inner product defined by the above norm varies if $\varphi$ is not constant with respect to $t$.
		Set $E:=\cup_{t\in U}E_t$ and view it as a trivial vector bundle over $U$ (of infinite rank) in a natural way, then $E$ is a holomorphic vector bundle with a Hermitian metric
        $h$ given by
		$$h_t(f,g):=\int_D f(z)\overline{g(z)} e^{-\varphi(t,z)}d\lambda(z),\ \forall\ f,g\in E_t.$$
		A fundamental result of Berndtsson is as follows		
		\begin{thm}[{\cite[Theorem 1.1]{Ber09}}]\label{thm:Bern direct image bd domain}
			With the above notations and assumptions, the curvature of $(E,h)$ is positive in the sense of Nakano,
			and is strictly positive in the sense of Nakano if $\varphi$ is strictly plurisubharmonic.
		\end{thm}
		To generalize Theorem \ref{thm:Bern direct image bd domain}, we introduce some notions and notations which have been used in \cite{DHQ}
         with a few slight modifications.\\
		\indent Let $p\colon \mc^n\times\mc^m\ra\mc^n$ be the natural projection. For a set $A\subset \mc^n\times\mc^m$,
		we denote  $A_t:=p^{-1}(t)\cap A$, which is called the fiber of $A$ over $t$. Of course we can view $A_t$ as a family of subsets in $\mc^m$ depending on the parameter $t$.
		
		\begin{defn}\label{def:strict p.s.c family}\
			Let $U\subset\mc^n$ be a domain. \bi
			\item[(1)] A \emph{family of bounded domains} (of dimension $m$ over $U$) is a domain $\Omega\subset U\times\mc^m$ such that $p(\Omega)=U$
			and all fibers $\Omega_t\subset\mc^m\ (t\in U)$ are bounded.
			\item[(2)] A family of bounded domains $\Omega\subset U\times\mc^m$  has \emph{ $C^k$ ($k\in\mathbb{N}\cup\{\infty\}$) boundary }if  
             there exists a $C^k$-function $\rho(t,z)$
			defined on $\mc^n\times\mc^m$ such that $\Omega=\{(t,z)\in U\times\mc^m|\ \rho(t,z)<0\}$ and $d(\rho|_{\Omega_t})\neq 0$ 
            on $\partial\Omega_t$ for any $t\in U$.
			Such a function $\rho$ is  called a \emph{(boundary) defining function} of $\Omega$.
			\item[(3)] A family of bounded domains $\Omega\subset U\times\mc^m$ with $C^4$-boundary is said to \emph{have a plurisubharmonic defining function} if it admits a 
            $C^4$-boundary   defining function  $\rho$ that is plurisubharmonic on some neighborhood of $\overline{\Omega}$ (the closure of $\Omega$) and $\rho(t,\cdot)$ is \emph{strictly plurisubharmonic} on $\overline\Omega_t$ for any $t\in U$.
            If moreover $\rho$ is smooth, then we say $\Omega$ has a smooth plurisubharmonic defining function.
            \item[(4)] A family of bounded domains $\Omega\subset U\times\mc^m$ is called \emph{strictly pseudoconvex} if it admits a $C^4$-boundary defining function that is  strictly plurisubharmonic on some neighborhood of $\overline{\Omega}$. 
			\ei
		\end{defn}
        \indent Then we may generalize Theorem \ref{thm:Bern direct image bd domain} to the following
        \begin{thm}\label{thm(extend):Bern direct image bd domain}
        Let $\Omega\subset U\times\mc^m$ be a family of bounded domains over a domain $U\subset\mc^n$ which has a plurisubharmonic defining function $\rho$,
        and let $(F,h^F)$ be a Hermitian holomorphic trivial vector bundle of finite rank $r$ defined on some neighborhood of $\overline\Omega$.
        Let $\{e_1,\cdots,e_r\}$ be the canonical holomorphic frame of $F$. For any $t\in U$, set
    	$$E_t:=\{u=\sum_{\lambda=1}^r u_\lambda dz_1\wedge\cdots\wedge dz_m\otimes e_\lambda|\ u_\lambda \in \mo(\Omega_t)\cap L^2(\Omega_t)\text{ for all }\lambda\}$$
    	with an inner product $h_t$ given by
    	$$h_t^E(u,v):=\int_{\Omega_t} \sum_{\lambda,\mu=1}^ru_\lambda(z)\overline{v_\mu(z)}h^F_{(t,z)}(e_\lambda,e_\mu)d\lambda(z),\ \forall\ u,v \in E_t.$$
        We set $E:=\cup_{t\in U}E_t$, then $(E,h^E)$ is Nakano positive if $(F,h^F)$ is Nakano positive on $\Omega$, and it is Nakano strictly positive if $(F,h^F)$ is Nakano strictly positive on $\Omega$.
        \end{thm}
         \indent For the definition of Nakano positivity (resp. strict positivity), see Definition \ref{Curvature positivity of Berndtsson}
         and Definition \ref{Curvature positivity 1}. To prove Theorem \ref{thm(extend):Bern direct image bd domain}, we construct a curvature formula of $E$. More explicitly, we get the following
          \begin{thm}[see Theorem \ref{curvature_formula}] \label{curvature formula 1}
          Assumptions (or notations) of $\Omega,F,E$ as in Theorem \ref{thm(extend):Bern direct image bd domain}. For any  $u_1,\cdots,u_n\in C^2(U,E)$, we have (locally) 
            \begin{align}\label{formula 1}
                &\quad h^E(\Theta_{jk}^{(E,h^E)}u_j,u_k)\nonumber\\
                &=\sum_{\lambda,\mu}\int_{\Omega_t}H(h^F)_{jk\lambda\mu}u_{j\lambda}\overline{u_{k\mu}}d\lambda-\int_{\Omega_t}h^F(\pi_{\perp}(L_ju_j),\pi_{\perp}(L_ku_k))d\lambda \\
                &\quad+\sum_{p,q}\int_{\Omega_t}\frac{\partial\nu_j^p}{\partial\bar{z}_q}\frac{\partial\overline{\nu_k^q}}{\partial z_p} h^F\left(u_j,u_k\right)d\lambda+\int_{\partial\Omega_t}H(\rho)_{j\bar{k}}h^F(u_j,u_k)\frac{dS}{|\nabla\rho|},\nonumber
               \end{align}
               where 
               $$
                L_ju_j:=D_{t_j}^Fu_j+\sum_p D_{z_p}^F(\nu_j^p u_j),
               $$
               $$
                 H(h^F)_{jk\lambda\mu}:=A_{jk\lambda\mu}+\sum_p\nu_j^pA_{pk\lambda\mu}+\sum_q\overline{\nu_k^q}A_{jq\lambda\mu}+\sum_{p,q}\nu_j^p\overline{\nu_k^q}A_{pq\lambda\mu},
               $$
               $$
                H(\rho)_{j\bar{k}}:=\left(\rho_{j\bar{k}}-\sum_{p,s}\rho_{j\bar{s}}\rho^{s\bar{p}}\overline{\rho_{k\bar{p}}}\right)
                +\frac{(\rho_j-\sum_s\rho_{j\bar{s}}\rho^s)(\rho_{\bar{k}}-\sum_p\overline{\rho_{k\bar{p}}}\rho^{\bar{p}})}{|\partial\rho|^2}.
          $$
            \end{thm}
        \begin{remark}
         Note that in the proof of Theorem \ref{curvature formula 1}, we didn't use the full strength that $\rho$ is a plurisubharmonic defining function,
         and we only use the fact that $\rho$ is a boundary defining function such that each $\rho(t,\cdot)$ is strictly plurisubharmonic on $\overline\Omega_t$
         for any $t\in U$. However, to get Theorem \ref{thm(extend):Bern direct image bd domain}, we must assume $\rho$ is a plurisubharmonic defining function.
        \end{remark}
        In particular, when $F$ is a trivial line bundle, we get 
        \begin{cor}\label{curvature formula:line bundle}
         Let $\Omega\subset U\times\mc^m$ be a family of bounded domains over a domain $U\subset\mc^n$ which has a plurisubharmonic defining function $\rho$,
        and let $(F,e^{-\varphi})$ be a Hermitian trivial line bundle defined on some neighborhood of $\overline\Omega$,
        where $\varphi\in C^2(\overline\Omega)$. Let $e$ be the canonical holomorphic frame of $F$. For any $t\in U$, set
    	$$E_t:=\{u= udz_1\wedge\cdots\wedge dz_m\otimes e|\ u\in \mo(\Omega_t)\cap L^2(\Omega_t)\}\cong \mo(\Omega_t)\cap L^2(\Omega_t)$$
    	with an inner product $h_t$ given by
    	$$h_t^E(u,v):=\int_{\Omega_t} u(z)\overline{v(z)}e^{-\varphi(t,z)}d\lambda(z),\ \forall u,v \in E_t,$$
        and set $E:=\cup_{t\in U}E_t.$ Then for any  $u_1,\cdots,u_n\in C^2(U,E)$, we have (locally) 
             \begin{align*}
                &\quad h^E(\Theta_{jk}^{(E,h^E)}u_j,u_k)\nonumber\\
                &=\sum_{\lambda,\mu}\int_{\Omega_t}H(\varphi)_{j\bar k}u_{j}\overline{u_{k}}e^{-\varphi}d\lambda-\int_{\Omega_t}\pi_{\perp}(L_ju_j)\overline{\pi_{\perp}(L_ku_k)}
                e^{-\varphi}d\lambda \\
                &\quad+\sum_{p,q}\int_{\Omega_t}\frac{\partial\nu_j^p}{\partial\bar{z}_q}\frac{\partial\overline{\nu_k^q}}{\partial z_p} u_j\overline{u_k}e^{-\varphi}d\lambda+\int_{\partial\Omega_t}H(\rho)_{j\bar{k}}u_j\overline{u_k}e^{-\varphi}\frac{dS}{|\nabla\rho|},\nonumber
               \end{align*}
               where $\pi_\perp\colon L^2(\Omega_t)\rw (\mo(\Omega_t)\cap L^2(\Omega_t))^{\perp}$ is the orthogonal projection with respect to the inner product 
               given by $e^{-\varphi}d\lambda$, and 
                $$
                L_ju_j:=\frac{\partial u_j}{\partial t_j}-u_j\varphi_j+\sum_p \left(\frac{\partial (\nu_j^p u_j)}{\partial z_p}-\nu_j^pu_j\varphi_p\right),
               $$
               $$
                 H(\varphi)_{j\bar k}:=\varphi_{j\bar k}+\sum_p\nu_j^p\varphi_{p\bar k}+\sum_q\overline{\nu_k^q}\varphi_{jq}+\sum_{p,q}\nu_j^p\overline{\nu_k^q}\varphi_{p\bar q},
               $$
               $$
                H(\rho)_{j\bar{k}}:=\left(\rho_{j\bar{k}}-\sum_{p,s}\rho_{j\bar{s}}\rho^{s\bar{p}}\overline{\rho_{k\bar{p}}}\right)
                +\frac{(\rho_j-\sum_s\rho_{j\bar{s}}\rho^s)(\rho_{\bar{k}}-\sum_p\overline{\rho_{k\bar{p}}}\rho^{\bar{p}})}{|\partial\rho|^2}.
               $$
        \end{cor}
        \indent All notations and related concepts in Theorem \ref{curvature formula 1} will be explained in Section \ref{section:curvature formula}. 
        When $F$ is a line bundle, the curvature formula has been constructed in \cite{Wang17}, see \cite[Theorem 2.11]{Wang17}.
        Different from the method using in \cite{Wang17} which depends strongly on Siu's theorem and Hamilton's theorem (see the proof of \cite[Lemma 4.7]{Wang17}), the method we adopt is more direct and we didn't make any additional requirements on the positivity of the vector bundle $F$ (comparing formula (2.26) of \cite{Wang17}). 
       After we construct the curvature formula for vector bundles, Theorem \ref{thm(extend):Bern direct image bd domain}
        is then a simple consequence of it. Note that to make things simple and to make the main idea clear, we only consider sections of $E$
        which are $C^2$-smooth up to the boundary of $\Omega$. To make the main idea of the proof of Theorem \ref{thm(extend):Bern direct image bd domain} clear,
        the case that $F$ is a line bundle, i.e., the proof of Corollary \ref{curvature formula:line bundle} is given in Section \ref{section:curvature formula}, and the general case is given in the appendix \ref{appendix}.\\
        \indent As mentioned in the abstract, our main purpose is to study curvature strict positivity of direct image bundles.  We state the main result of this paper as follows
        \begin{thm}\label{thm:vector bundles of holomorphic functions}
         Assumptions (or notations) of $\Omega,F,E$ as in Theorem \ref{thm(extend):Bern direct image bd domain}.
          We moreover assume $\Omega$ is a strictly pseudoconvex family of bounded domains over $U$, then $(E,h^E)$ is strictly positive in the sense of Nakano 
         if $(F,h^F)$ is just Nakano positive on $\Omega$.
        \end{thm}
        Note that we didn't assume $(F,h^F)$ is Nakano strictly positive, so we may think the curvature strict positivity of $(E,h^E)$ comes from 
        the strict pseudoconvexity of $\Omega$.\\
        \indent Before we state the main idea of the proof of Theorem \ref{thm:vector bundles of holomorphic functions}, let us first give some applications 
        of it to complex analysis and convex analysis.\\
        \indent By Theorem \ref{thm:vector bundles of holomorphic functions}, Corollary 1.4 in \cite{DHQ} holds for an arbitrary  strictly pseudoconvex family of bounded domains, without symmetry. The proof is similar to that proof, so we omit it here.
        \begin{cor}
         Let $\Omega\subset U\times\mc^m$ be a strictly pseudoconvex family of bounded domains  over a domain $U\subset\mc^n$ and let $\varphi
         \in C^0(\overline\Omega)\cap C^2(\Omega)$ be a plurisubharmonic function on $\Omega$. For any $t\in U$, let $K(t,z)$ be the weighted Bergman kernel of $\Omega_t$ with weight $\varphi(t,\cdot),$ then $\ln K(t,z)$ is a strictly plurisubharmonic function on $\Omega$.
        \end{cor}
       \indent  Recall a domain $D\subset\mc^m$ is called a circular domain if it is invariant under the action of $\mathbb{S}^1$ given by 
        $$e^{i\theta}\cdot(z_1,\cdots,z_m)=\left(e^{i\theta}z_1,\cdots,e^{i\theta}z_m\right),\ \forall\ \theta\in\mr,$$
        and is called a Reinhardt domain if it is invariant under the action of $\mathbb{T}^m$ given by 
        $$\left(e^{i\theta_1},\cdots,e^{i\theta_m}\right)\cdot (z_1,\cdots,z_m)=\left(e^{i\theta_1}z_1,\cdots,e^{i\theta_m}z_m\right),\ \forall\ \theta_1,\cdots,\theta_m\in\mr.$$
        It is clear that \cite[Theorem 1.4]{DHJQ} is a simple consequence of Theorem \ref{thm:vector bundles of holomorphic functions}
        (see \cite[Lemma 3.1]{DHQ} for the details). Note that we don't require $(F,h^F)$ is Nakano positive on some neighborhood of $\overline\Omega$ in Theorem \ref{thm:vector bundles of holomorphic functions}. However, the requirement $(F,h^F)$ is Nakano positive on some larger domains is indispensable in the proof of  \cite[Theorem 1.4]{DHJQ}.
        For convenience, we state the precise statement as follows. The proof is a simple modification of \cite[Lemma 3.1]{DHQ}, so we won't give it.
        \begin{cor}\label{qin:application}
        Let $\Omega\subset U\times\mc^m$ be a strictly pseudoconvex family of bounded domains over a domain $U\subset\mc^n$
        and let $(F,h^F)$ be a Hermitian trivial vector bundle of finite rank $r$ defined on some neighborhood of $\overline\Omega$.
        Let $\{e_1,\cdots,e_r\}$ be the canonical holomorphic frame of $F$.
		We assume that all fibers $\Omega_t\ (t\in U)$ are (connected) Reinhardt domains and $h^F_{(t,z)}$ is $\mathbb{T}^n$-invariant with respect to $z$ for any $t\in U$.
        For any $t\in U$, set $E_t:=\mc^r$, with an inner product given by 
        $$h_t^E(u,v):=\int_{\Omega_t}\sum_{\lambda,\mu=1}^ru_\lambda\overline{v_\mu} h^F_{(t,z)}(e_\lambda,e_\mu)d\lambda(z),\ \forall\ u,v\in E_t.$$
        We set $E:=\cup_{t\in U}E_t$ and view it as a Hermitian holomorphic (trivial) vector bundle over $U$, then the curvature of $(E,h^E)$ 
        is  strictly positive in the sense of Nakano if $(F,h^F)$ is Nakano positive on $\Omega$.
        \end{cor}
       \indent Of course, similar to Theorem 1.2 and Theorem 1.3 of \cite{DHJQ}, we may replace $E_t$ by the space of homogeneous polynomials of the same degree,
       or we can replace Reinhardt domains by circular domains. More generally, it is also possible to consider any (smooth) action of compact Lie group
       and the corresponding irreducible unitary sub-representations according to the well known Peter-Weyl theorem. However, we didn't pursue this here.\\
        \indent Using Corollary \ref{qin:application} and the same method in the proof of \cite[Corollary 1.5]{DHJQ},  we can get the curvature strict positivity of a strict convex family of 
        bounded domains (the definition is similar to Definition \ref{def:strict p.s.c family}, see \cite[Page 6]{DHQ} or \cite[Definition 2.3]{DHJQ}).	
		 \begin{cor}\label{qin:convex}
			 Let $D\subset U_0\times\mr^m$ be a  strictly convex family of bounded domains (with connected fibers) over a domain $U_0\subset\mr^n,$
            and let $(F,g^F)$ be a trivial vector bundle of finite rank $r$ defined on some neighborhood  of $\overline D$.
            Let $\{e_1,\cdots,e_r\}$ be the canonical frame of $F$.
            For any $t\in U_0$, set $E_t:=\mr^r$, with an inner product $g_t^E$ given by 
            $$g_t^E(u,v):=\int_{D_t}\sum_{\lambda,\mu=1}^ru_\lambda v_\mu g^F_{(t,x)}(e_\lambda,e_\mu)d\lambda(x)$$
            for all 
            $$u:=(u_1,\cdots,u_r)\in E_t,\ v:=(v_1,\cdots,v_n)\in E_t.$$
            We set $E:=\cup_{t\in U_0}E_t$ and view it as a Riemannian (trivial) vector bundle over $U_0$, then the curvature of $(E,g^E)$ 
            is strictly positive in the sense of Nakano if $(F,g^F)$ is Nakano positive on  $D$. 
		\end{cor}
         \indent Corollary \ref{qin:convex} is a corresponding result  of Pr\'ekopa's Theorem.  For the classical version of Pr\'ekopa's Theorem, see \cite{Pre73},
         and see \cite{Rau13} for the matrix valued version.  Note that different from that two papers, we here emphasize the strict positivity of $(E,g^E)$ when $D$ is a strictly convex family of bounded domains.\\  
        \indent Now we state the main idea of the proof of Theorem \ref{thm:vector bundles of holomorphic functions}.
        In the curvature formula (\ref{formula 1}), after we summing all $j,k$, it is clear that the first three terms in the right hand side is nonnegative,
        and the last term will produce a positive lower bound by the following simple observation, which is what we wanted.
         \begin{lem}\label{lem:basic inequality for holomorphic functions}
         \cite[Corolllary 1.7]{DJQ}
        Let $D\subset\mc^m$ be a bounded domain with $C^1$-boundary, then there is a constant $\delta:=\delta(D)>0$ such that 
       $$
       \delta\int_{D}|f|^2d\lambda\leq \int_{\partial D}|f|^2dS,\ \forall f\in \mo(D)\cap C^0(\overline{D}),
       $$
        where $dS$ is the surface measure of $\partial D$, $C^0(\overline{D})$ denotes the space of complex valued functions that are continuous to the boundary of $D$.
        \end{lem}
        \indent Lemma \ref{lem:basic inequality for holomorphic functions} plays an essential role in the proof of Theorem \ref{thm:vector bundles of holomorphic functions}, see \cite{DGQ} and \cite{DJQ} for more related inequalities. When $\Omega$ is strictly pseudoconvex, Lemma \ref{lem:basic inequality for holomorphic functions} have appeared implicitly in the proof of \cite[Proposition 2.2.2]{Ber95}.
        \begin{remark}
            By Lemma 6.6 of \cite{DJQ}, we may give a concrete value of the above $\delta$, and then we can give a concrete lower bound for the curvature strict positivity  of $(E,h^E)$,
            which only depends on $|\Omega_t|,|\partial\Omega_t|,m$ and $\rho$. 
        \end{remark}
        \indent Motivated by the reviewer of the first version of this paper, we further find a flatness criterion for the direct image bundles as a consequence of the curvature formula
        in Theorem \ref{curvature formula 1}. To state it, let us recall the notion of triviality of a family of bounded domains. 
        \begin{defn}\cite[Definition 2.16]{Wang17}\label{defn:flat}
         Let $p\colon \Omega\subset U\times\mc^m\rw U$ be a family of bounded domains which has a smooth plurisubharmonic defining function over a domain $U\subset\mc^n$, then we say $\Omega$ is {\it locally trivial} if 
         for any $t_0\in  U$ , there are an open neighborhood $U_0\subset U$ of $t_0$  and  a biholomorphic map $\Phi\colon U_0\times\Omega_{t_0}\rw p^{-1}(U_0)$ such that $\Phi$ can be extended smoothly to a neighborhood of $U_0\times \overline{\Omega_{t_0}}$ and  
        $$\Phi(\{t\}\times \Omega_{t_0})=\Omega_t,\ \forall\ t\in U_0.$$ 
        \end{defn} 
        \indent In Definition 2.16 of \cite{Wang17}, different from our requirements that $\Phi$ can be extended smoothly to a neighborhood of  $U_0\times \overline{\Omega_{t_0}},$
         the author just requires that $\Phi_*(\frac{\partial}{\partial t_j})$ can be extended smoothly to the boundary  of $\Omega$ for any $1\leq j\leq n$.
         We don't know whether the two definitions are equivalent or not, but Definition \ref{defn:flat} is more convenient in the proof 
         of the part (2) implies (1) of Theorem \ref{thm:flatness criterion}.\\
        \indent Now we can state a flatness criterion for the direct image bundles.
        \begin{thm}\label{thm:flatness criterion}
         Let $\Omega\subset U\times\mc^m(m\geq 2)$ be a family of bounded domains which has a plurisubharmonic defining function $\rho$ over a domain $U\subset\mc^n,$
        and let $(F,h^F)$ be a Hermitian holomorphic trivial vector bundle of finite rank $r$ defined on some neighborhood of $\overline\Omega$
        which is Nakano positive on $\Omega$. Define $E$ and $h^E$ as in Theorem \ref{thm(extend):Bern direct image bd domain},  then the following two statements are equivalent:
        \begin{itemize}
          \item[(1)] $\Theta^{(E,h^E)}\equiv 0$ on $U.$
          \item[(2)] $\Theta^{(F,h^F)}\equiv 0$ on $\Omega$ and $\Omega$ is locally trivial.
        \end{itemize}
        \end{thm}
         We remark that in \cite[Theorem 2.17]{Wang17}, the author assume $F$ is flat to get the triviality of $\Omega$ from the flatness 
         of $E$. However, this assumption is redundant, so we only assume $F$ is Nakano positive in Theorem \ref{thm:flatness criterion}. The main idea of the proof of Theorem \ref{thm:flatness criterion} is similar with the proof of \cite[Theorem 2.17]{Wang17} and we will only give a sketch of it in Section \ref{section:criterion}.
        \subsection*{Acknowledgements}
    		The author is grateful to Professor Fusheng Deng, his Ph.D. advisor, for valuable discussions on related topics,
            and Professor Xu Wang for answering some questions about the paper \cite{Wang17}. The author also thanks 
            the referees for their valuable suggestions.
		\section{Preliminaries}\label{sec:preliminary}
		In this section, we fix some notations, conventions and  collect some knowledge that are needed in our discussions.
        \subsection{Notations and conventions}\ \\
        \indent Our convection for $\mathbb{N}$ is that $\mathbb{N}:=\{1,2,3,\cdots\}$.
        For any subset $A\subset\mr^n$, let $\overline{A}$ denote the closure of $A$ in $\mr^n$, and let $\partial A:=\overline{A}\setminus A$.
        Any neighborhood of a subset $A\subset\mr^n$ is open by default. We say $U$ is a domain of $\mr^n$ if it is a connected
        open subset of $\mr^n$. \\
        \indent   For any $k\in\mathbb{N}\cup\{0,\infty\}$.  If $U\subset \mr^n$ is an open subset, then we let $C^k(U)$ be the space of complex valued functions which are of class $C^k$. For any $f\in C^1(U)$, let $\nabla f$ denote its gradient. If $A\subset\mr^n$ is an arbitrary subset, let  $C^k(A)$ denote the space of complex valued
         functions which is of class $C^k$ in an open neighborhood of $A$.\\
        \indent If $U\subset\mc^n$ is a open subset, let $\mo(U)$ denote the space of holomorphic functions in $U$. 
        We recall an upper semi-continuous function $f\colon U\rw [-\infty,\infty)$
        is plurisubharmonic if its restriction to every complex line in $U$ is subharmonic, i.e. satisfying the submean value inequality.
        A real valued function $\varphi\in C^2(U)$ is strictly plurisubharmonic if its complex Hessian $\left(\varphi_{j\bar{k}}\right)_{1\leq j,k\leq n}$ is positive definite for any $z\in U$, where 
        $$\varphi_{j\bar{k}}:=\frac{\partial^2\varphi}{\partial z_j\partial z_k},\ 1\leq j,k\leq n,$$ 
        and we let $\left(\varphi^{j\bar{k}}\right)_{1\leq j,k\leq n}$ be the inverse matrix of its complex Hessian.
        We recall that $U$ is strictly pseudoconvex if it has $C^2$-boundary, and it has a $C^2$-boundary defining function which is strictly plurisubahrmonic in 
        a neighborhood of $\overline{U}$.\\
        \indent Let $d\lambda$ denote the Lebesgue measure on $\mr^n$. Note that all measurable sets, measurable maps, measurable differential forms
        will be Lebesgue measurable.  If an open subset $U\subset \mr^n$ has $C^1$-boundary, let $dS$ denote the surface measure on $\partial U$.\\
         \indent For any open subset $U\subset\mr^n$, let $L^2(U)$ denote the set of $L^2$-integrable complex valued functions. For any $f\in L^2(U)$, let $\|f\|_{L^2(U)}$ denote the $L^2$-norm of $f$.
         \\
         \indent  A metric on a vector bundle will be of class $C^2$, all sections of a vector bundle will be global unless stated otherwise. Let $(E,h^E)$ be a  Hermitian holomorphic vector bundle of finite rank $r$  over an open subset $U\subset\mc^n$, and let  $e_1,\cdots,e_r$ be a local  holomorphic frame of $E$.  For any $x\in U$, let $h^E_x$ denote the inner product on the fiber $E_x$ of $x$. The metric $h^E$ will always be complex linear with respect to its first entry.  For any $1\leq \alpha,\beta\leq r,$ set 
         $$h^E_{\alpha\bar{\beta}}:=h^E(e_\alpha,e_\beta),$$
         and let $\left((h^E)^{\alpha\bar{\beta}}\right)_{1\leq\alpha,\beta\leq n}$ be the inverse matrix of 
         $\left(h^E_{\alpha\bar{\beta}}\right)_{1\leq \alpha,\beta\leq r}.$
         Let $u,v$ be any measurable sections of $E$,  write 
         $$u=\sum_{\lambda=1}^ru_\lambda e_\lambda,\ v=\sum_{\mu=1}^rv_\mu e_{\mu},$$
         then we let 
         $$|u|^2:=\sum_{\lambda=1}^r|u_\lambda|^2,\ \langle u,v\rangle_{h^E}:=h^E(u,v)=\sum_{\lambda,\mu=1}^ru_\lambda\overline{v_\mu}h^E(e_\lambda,e_\mu),$$
         $$|u|_{h}^2:=\langle u,u\rangle_{h},\ (u,v)_{h}:=\int_{U}\langle u,v\rangle_{h^E} d\lambda,\ \|u\|_{h^E}^2:=(u,u)_{h^E}.$$
         Let $L^2(U,E)$ be the space of measurable $E$-valued sections which are $L^2$-integrable, i.e.
         $$L^2(U,E):=\{u\text{ is a measurable } E\text{-valued section of } E|\ \|u\|_{h^E}^2<\infty\}.$$
         For any $k\in\mathbb{N}\cup\{0,\infty\}$, let $C^k(U,E)$ be the space of $E$-valued sections which are of class $C^k$.
		\subsection{Curvature positivity of Hermitian holomorphic vector bundles}\

		\indent Let $U\subset \mc_t^n$ be a domain, and let $(E,h^E)$ be a Hermitian holomorphic trivial vector bundle over $U$ of finite rank $r$.
		The Chern connection is given by a collection of differential operators $\{D_{t_j}^E\}_{1\leq j\leq n}$, which satisfies 
		$$
        \partial_{t_j}h^E(u, v) = h^E(D_{t_j}^E u, v) + h^E(u, \bar{\partial}_{t_j} v),\ \forall\ u,v \in C^2(U,E),
        $$
		where $\partial_{t_j}:= \frac{\partial}{\partial t_j}$ and $\bar{\partial}_{t_j}:= \frac{\partial}{\partial \bar{t}_j}$
        for any $1\leq j\leq n.$ For any $1\leq j,k\leq n,$ let $\Theta^{(E,h^E)}_{jk}:=[D_{t_j}^E, \bar{\partial}_{t_k}]$ (Lie bracket), then the Chern curvature of 
       $(E,h^E)$ is given by 
		$$
         \Theta^{(E,h^E)} =\sum_{j,k=1}^n \Theta^{(E,h^E)}_{jk} dt_j \wedge d\bar{t}_k.
        $$
		\begin{defn}\label{Curvature positivity of Berndtsson}
			The curvature of $(E,h^E)$ is said to be positive (resp. strictly positive) in the sense of Nakano if for any nonzero $n$-tuple $(u_1, \cdots  , u_n)\\
           \in C^\infty(U,E),$
            we have 
			$$
            \sum_{j,k=1}^n h^E\left(\Theta_{jk}^{(E,h^E)} u_j, u_k\right) \geq 0\ (\text{resp.} >0).
            $$
		\end{defn}
          \section{Curvature formulas of direct image bundles}\label{section:curvature formula}\
          In this section, we will construct curvature formulas for direct image bundles. To do this, let us fix some notations 
          and introduce some basic concepts.\\
          \indent Let $\Omega\subset U\times\mc^m$ be a family of bounded domains which has a plurisubharmonic defining function over a domain $U\subset\mc^n,$
             let $\rho$ be a $C^4$-boundary defining function of it, and let $(F,h^F)$ be a Hermitian holomorphic trivial vector bundle of finite rank $r$ defined on some neighborhood of $\overline\Omega$. Let $j,k=1,\cdots, n$ represent the indices of the components of $t=(t_1,\cdots, t_n)$,
            $p,q,s,a,b,c=1, \cdots, m$ represent the indices of the components of  $z=(z_1,\cdots, z_m)$, $\lambda,\mu,\alpha,\beta,\gamma=1,\cdots, r$ represent the indices of the components of $F$,
            and let $\{e_1,\cdots,e_r\}$ be the canonical holomorphic frame of $F$. For any $t\in U$, set
    		$$E_t:=\{u:=\sum_{\lambda} u_\lambda dz_1\wedge\cdots\wedge dz_m\otimes e_\lambda|\ u_\lambda \in \mo(\Omega_t)\cap L^2(\Omega_t)\text{ for all }\lambda\}$$
    		with an inner product $h_t^E$ given by
    		$$h_t^E(u,v):=\int_{\Omega_t} \sum_{\lambda,\mu}u_\lambda(z)\overline{v_\mu(z)}h^F_{(t,z)}(e_\lambda,e_\mu)d\lambda(z),\ \forall u,v \in E_t.$$
    		We set $E:=\cup_{t\in U}E_t$, and view it as a Hermitian holomorphic (trivial) vector bundle over $U$
           in some sense (which we don't pursue in this direction). To simplify the notation, we always write $dz$ for $dz_1\wedge\cdots\wedge dz_m$.
           \begin{defn}
            We call $u\colon U\rw E,\ t\rwo u(t):=\sum_{\lambda=1}^r u_\lambda(t,z)dz\otimes e_\lambda\in E_t$ a $C^k$-section ($0\leq k\leq \infty$) of $E$ if 
            $u_\lambda\in C^k(\overline{\Omega},\mc)$ for any $1\leq \lambda\leq r$. Let $C^k(U,E)$ denote the space of all $C^k$-sections of $E$. 
           \end{defn}
           \indent Note that any section of $E$ can be naturally regarded as a section of $F.$\\
            \indent To give the definition of Chern curvature, let us recall some knowledge that have been used in \cite{Wang17}.
            For a $C^1$ function $\varphi$ on $\Omega$, we set 
            $$\varphi_j:=\frac{\partial\varphi}{\partial t_j},\ \varphi_{\bar j}:=\frac{\partial\varphi}{\partial \bar{t}_j},\ \varphi_p:=\frac{\partial\varphi}{\partial z_p},
            \ \varphi_{\bar p}:=\frac{\partial\varphi}{\partial \bar{z}_p}.$$
            We moreover set 
             $$
            \rho^{p}:=\sum_q\rho^{p\bar{q}}\rho_{q},\ \rho^{\bar{p}}:=\overline{\rho^{p}},\ |\partial\rho|^2:=\sum_{p}\rho_{\bar p}\rho^p,$$
            $$\nu_j^p:=\frac{\rho_j\rho^{\bar p}}{\rho-|\partial\rho|^2}-\sum_q\left(\rho_{j\bar{q}}\rho^{q\bar{p}}+\frac{\rho_{j\bar{q}}\rho^{q}\rho^{\bar p}}{\rho-|\partial\rho|^2}\right),  $$
            $$
            V_j:=\frac{\partial}{\partial t_j}+\sum_p \nu_j^p\frac{\partial}{\partial z_p},\ V_{\bar{j}}:=\overline{V_{j}}.
            $$
            It is clear  that $\nu_j^p\in C^2(\overline\Omega)$. A direct computation shows that  $
            V_j(\rho)|_{\partial\Omega}=0.$ By Corollary A.4 and Lemma 2.9 (and its proof) of \cite{Wang17}, we have
            \begin{lem}\label{lem:derivative}
            Let $f\in C^1(\overline \Omega)$, then for any $t_0\in U$, there is an open neighborhood $V\subset U$ of $t_0$ such that, for any $t\in V$, we have 
            $$
            \frac{\partial}{\partial t_j}\int_{\Omega_t}f(t,z)d\lambda(z)=\int_{\Omega_t}\left(\frac{\partial f}{\partial t_j}(t,z)+
            \sum_p\frac{\partial}{\partial z_p}(\nu_j^p f)(t,z)\right)d\lambda(z).
            $$
            \end{lem}
              \indent Lemma \ref{lem:derivative} is the starting point for the construction of curvature formulas.\\
             \indent We also recall the definition of $\bar{\partial}$-operator and Chern connection of $E$. For $u\in C^1(U,E)$, write 
             $$
             u:=\sum_{\lambda}u_\lambda dz\otimes e_\lambda,\ \bar{\partial}_{t_j}u:=\frac{\partial u_\lambda}{\partial \bar{t}_j}dz\otimes e_\lambda,
             $$
             then 
             $$
             \bar{\partial}u:=\sum_{j=1}^n d\bar{t}_j\wedge \bar{\partial}_{t_j}u
             $$
             is the $\bar{\partial}$-operator of $E$. 
             \begin{defn}
             The Chern connection on $E$ is a collection of differential operators $\{D_{t_j}^E\}_{1\leq j\leq n}$ on $E$
             such that for each $j$, we have  
    		$$
            \partial_{t_j} h^E(u, v) = h^E(D_{t_j}^E u, v) + h^E(u, \bar{\partial}_{t_j} v),\ \forall\ u,v \in C^1(U,E).
            $$
             \end{defn}
              For any $t\in U$, set 
            $$G_t:=\{u=\sum_{\lambda} u_\lambda dz\otimes e_\lambda|\ u_\lambda \in L^2(\Omega_t)\text{ for all }\lambda\}$$
    		with an inner product $h^G_t$ given by
    		$$h_t^G(u,v):=\int_{\Omega_t} \sum_{\lambda,\mu}u_\lambda(z)\overline{v_\mu(z)}h^F_{(t,z)}(e_\lambda,e_\mu)d\lambda(z),\ \forall\ u,v \in G_t,$$
    		then $G:=\cup_{t\in U}G_t$  can be viewed as a Hermitian holomorphic (trivial) vector bundle over $U.$ Let $\pi\colon G\rw E$ be the fiberwise orthogonal projection, and 
            $\pi_\perp$ be the orthogonal projection on the orthogonal complement of $E$ in $G.$\\
            \indent  Now we prove the existence and uniqueness of the Chern connection. 
             \begin{thm}\label{Chern}
              There exists a unique Chern connection on $E$, and it is (locally) given by 
            $$
             D_{t_j}^Eu=\pi\left(D_{t_j}^Fu+\sum_{p} \nu_j^p D_{z_p}^Fu+\sum_{p}\frac{\partial \nu_{j}^p}{\partial z_p}u\right),\ \forall\ u\in C^1(U,E),
            $$ 
            where if $u=\sum_{\lambda}u_\lambda dz\otimes e_\lambda$, then
            $$
                D_{z_p}^Fu:=\sum_{\lambda}\frac{\partial u_\lambda}{\partial z_p}dz\otimes e_\lambda+\sum_{\lambda}u_\lambda dz\otimes D_{z_p}^Fe_\lambda.
            $$
             \end{thm}
             \begin{proof}
             By Lemma \ref{lem:derivative}, we know 
             \begin{align*}
            &\quad h^E(D_{t_j}^Eu,v)+h^E(u,\frac{\partial v}{\partial\bar{t}_j})\\
            &=\frac{\partial }{\partial t_j}h^E(u,v) \\
            &=\int_{\Omega_{t}}\frac{\partial }{\partial t_j}h^F(u,v)d\lambda+\sum_p\int_{\Omega_t}\frac{\partial}{\partial z_p}(\nu_j^p h^F(u,v))d\lambda \\
            &=\int_{\Omega_t}h^F(D_{t_j}^Fu,v)d\lambda +\int_{\Omega_t}h^F(u,\frac{\partial v}{\partial\bar{t}_j})d\lambda
            +\int_{\Omega_t} h^F(D_{z_p}^F(\nu_j^p u),v)d\lambda .
            \end{align*}
            \end{proof}
             Set $\Theta^{(E,h^E)}_{jk}:=[D_{t_j}^E, \bar{\partial}_{t_k}]$, then the Chern curvature of 
           $(E,h^E)$ is given by 
    		$$
             \Theta^{(E,h^E)} :=\sum_{j,k=1}^n \Theta^{(E,h^E)}_{jk} dt_j \wedge d\bar{t}_k,
            $$
            and then we give the definition of Nakano positivity of $(E,h^E)$, this is similar to Definition \ref{Curvature positivity of Berndtsson}.
            \begin{defn}\label{Curvature positivity 1}
    			The curvature of $(E,h^E)$ is said to be positive (resp. strictly positive) in the sense of Nakano if for any nonzero $n$-tuple $(u_1, \cdots  , u_n)\\
              \in C^\infty(U,E),$
                we have 
    			$$
                \sum_{j,k=1}^n h^E(\Theta_{jk}^{(E,h^E)} u_j, u_k) \geq 0\ (\text{resp.} >0).
                $$
    		\end{defn}
            To compute the Chern curvature of $(E,h^E)$, we adopt the following notations
                $$
                D_{t_j}^Fe_\lambda=\sum_{\mu}\Gamma_{j\lambda}^\mu e_\mu,\ D_{z_p}^Fe_\lambda=\sum_{\mu}\Gamma_{p\lambda}^\mu e_\mu,
                $$
    	$$
    			\left(\begin{array}{cc}
    				(A_{pq\lambda}^\alpha)& (A_{pj\lambda}^\alpha)\\
    				(A_{jp\lambda}^\alpha) & (A_{jk\lambda}^\alpha) 
    			\end{array}\right):=\left(\begin{array}{cc}
    				\left(-\frac{\partial\Gamma_{p\lambda}^\alpha}{\partial\bar{z}_q}\right)& \left(-\frac{\partial\Gamma_{p\lambda}^\alpha}{\partial\bar{t}_j}\right)\\
    				\left(-\frac{\partial\Gamma_{j\lambda}}{\partial\bar{z}_p}\right)& \left(-\frac{\partial\Gamma_{j\lambda}^\alpha}{\partial\bar{t}_k}\right)
    			\end{array}\right),
    	$$
         $$
    			\left(\begin{array}{cc}
    				(A_{pq\lambda\mu})& (A_{pj\lambda\mu})\\
    				(A_{jp\mu\lambda}) & (A_{jk\lambda\mu}) 
    			\end{array}\right):=\left(\begin{array}{cc}
    				\left(-\frac{\partial\Gamma_{p\lambda}^\alpha}{\partial\bar{z}_q}h^F_{\alpha\mu}\right)& \left(-\frac{\partial\Gamma_{p\lambda}^\alpha}{\partial\bar{t}_j}h^F_{\alpha\mu}\right)\\
    				\left(-\frac{\partial\Gamma_{p\lambda}^\alpha}{\partial\bar{t}_j}h^F_{\alpha\mu}\right)^*& \left(-\frac{\partial\Gamma_{j\lambda}^\alpha}{\partial\bar{t}_k}h^F_{\alpha\mu}\right)
    			\end{array}\right),
         $$
                where $^{*}$ represents the conjugate transpose of a matrix. It is clear that 
    			$$\Theta_{jk}^{(F,h^F)}=[D_{t_j}^F,\bar{\partial}_{t_k}]=\sum_{\lambda,\alpha}A_{jk\lambda}^\alpha e^\lambda\otimes e_\alpha,$$
                where $e^1,\cdots,e^n$ is the dual basis of $e_1,\cdots,e_n$. \\
            \indent Now we can state and prove the curvature formula. As said in the introduction, the line bundle case is given in the following, and the general case  is given in the appendix \ref{appendix}.
            \begin{thm}[= Corollary \ref{curvature formula:line bundle}]\label{thm:line bundle}
         Assume $F$ is a trivial line bundle, and $h^{F}=e^{-\varphi}$ for some $\varphi\in C^2(\overline\Omega)$.
         Let $e$ be the canonical holomorphic frame of $F$. For any $t\in U$, set
    	$$E_t:=\{u= udz\otimes e|\ u\in \mo(\Omega_t)\cap L^2(\Omega_t)\}\cong \mo(\Omega_t)\cap L^2(\Omega_t)$$
    	with an inner product $h_t$ given by
    	$$h_t^E(u,v):=\int_{\Omega_t} u(z)\overline{v(z)}e^{-\varphi(t,z)}d\lambda(z),\ \forall u,v \in E_t,$$
        and set $E:=\cup_{t\in U}E_t.$ Then for any  $u_1,\cdots,u_n\in C^2(U,E)$, we have (locally) 
            \begin{align*}
                &\quad h^E(\Theta_{jk}^{(E,h^E)}u_j,u_k)\nonumber\\
                &=\sum_{\lambda,\mu}\int_{\Omega_t}H(\varphi)_{j\bar k}u_{j}\overline{u_{k}}e^{-\varphi}d\lambda-\int_{\Omega_t}\pi_{\perp}(L_ju_j)\overline{\pi_{\perp}(L_ku_k)}
                e^{-\varphi}d\lambda \\
                &\quad+\sum_{p,q}\int_{\Omega_t}\frac{\partial\nu_j^p}{\partial\bar{z}_q}\frac{\partial\overline{\nu_k^q}}{\partial z_p} u_j\overline{u_k}e^{-\varphi}d\lambda+\int_{\partial\Omega_t}H(\rho)_{j\bar{k}}u_j\overline{u_k}e^{-\varphi}\frac{dS}{|\nabla\rho|},\nonumber
               \end{align*}
               where 
               $$
                L_ju_j:=\frac{\partial u_j}{\partial t_j}-u_j\varphi_j+\sum_p \left(\frac{\partial (\nu_j^p u_j)}{\partial z_p}-\nu_j^pu_j\varphi_p\right),
               $$
               $$
                 H(\varphi)_{j\bar k}:=\varphi_{j\bar k}+\sum_p\nu_j^p\varphi_{p\bar k}+\sum_q\overline{\nu_k^q}\varphi_{jq}+\sum_{p,q}\nu_j^p\overline{\nu_k^q}\varphi_{p\bar q},
               $$
               $$
                H(\rho)_{j\bar{k}}:=\left(\rho_{j\bar{k}}-\sum_{p,s}\rho_{j\bar{s}}\rho^{s\bar{p}}\overline{\rho_{k\bar{p}}}\right)
                +\frac{(\rho_j-\sum_s\rho_{j\bar{s}}\rho^s)(\rho_{\bar{k}}-\sum_p\overline{\rho_{k\bar{p}}}\rho^{\bar{p}})}{|\partial\rho|^2}.
               $$
           In particular, the curvature operator $\Theta_{jk}^{(E,h^E)}$ is pointwise defined, i.e. it does not  involve derivatives with respect to $t\in U$.
            \end{thm}
           \begin{proof}
          Fix $t\in U$, and $j,k$. For any $u\in L^2(\Omega_t)$, we always identify it with a section $udz\otimes e$ of $G$
          if there does not cause some confusion.  For an element $u\in L^2(\Omega_t)$, we may give an orthogonal decomposition of $u$ as follows: 
          $$u=v_1+v_2,\ v_1\in \mo(\Omega_t)\cap L^2(\Omega_t),\ \int_{\Omega_t}v_1\overline{v_2}e^{-\varphi}d\lambda=0,$$
          then we know 
          $$\pi(u)=\pi(udz\otimes e)=v_1dz\otimes e=v_1,\ \pi_{\perp}(u)=\pi_{\perp}(udz\otimes e)=v_2dz\otimes e=v_2.$$
          For any $u\in C^1(U,F)$, set  
        $$
         N_j u:=D_{z_p}^F(\nu_j^p u)=\sum_p\nu_j^p\left(\frac{\partial u}{\partial z_p}-u\varphi_p\right)+\sum_p\frac{\partial \nu_{j}^p}{\partial z_p}u.
        $$
        By Lemma \ref{lem:derivative}, we have 
        \begin{align*}
        &\quad \frac{\partial }{\partial\bar{t}_k}\int_{\Omega_t}h^F(L_ju_j,u_k)d\lambda\\ 
        &=\int_{\Omega_t}\frac{\partial }{\partial\bar{t}_k} h^F(L_ju_j,u_k)d\lambda+\sum_q\int_{\Omega_t}\frac{\partial }{\partial \bar{z}_q}(\overline{\nu_k^q}h^F(L_ju_j,u_k))d\lambda \\
        &=\int_{\Omega_t}h^F\left(\frac{\partial}{\partial\bar{t}_k}(L_ju_j),u_k\right)d\lambda+\int_{\Omega_t}h^F(L_ju_j,D_{t_k}^Fu_k)d\lambda \\
        &\quad +\sum_q\int_{\Omega_t}h^F\left(\frac{\partial}{\partial\bar{z}_q}(L_ju_j),\nu_k^q u_k\right)d\lambda+
        \sum_q\int_{\Omega_t}h^F(L_ju_j, D_{z_q}^F(\nu_k^q u_k))d\lambda \\
        &=\int_{\Omega_t}h^F\left(\frac{\partial}{\partial\bar{t}_k}(L_ju_j),u_k\right)d\lambda+\int_{\Omega_t}h^F(L_ju_j,L_ku_k)d\lambda \\
        &\quad +\sum_q\int_{\Omega_t}h^F\left(\frac{\partial}{\partial\bar{z}_q}(L_ju_j),\nu_k^q u_k\right)d\lambda ,
        \end{align*}
        so we get 
         \begin{align*}
       &\quad h^E\left(\Theta_{jk}^{(E,h^E)}u_j,u_k\right)\\ 
       &=h^E\left(D_{t_j}^E\left(\frac{\partial u_j}{\partial\bar{t}_k}\right), u_k\right)-h^E\left(\frac{\partial}{\partial\bar{t}_k}(D_{t_j}^Eu_j),u_k\right) \\
       &=h^E\left(D_{t_j}^E\left(\frac{\partial u_j}{\partial\bar{t}_k}\right), u_k\right)-\frac{\partial }{\partial\bar{t}_k}h^E(D_{t_j}^Eu_j,u_k)
        +h^E(D_{t_j}^Eu_j,D_{t_k}^Eu_k)   \\
        &=h^E\left(\pi\left(L_j\left(\frac{\partial u_j}{\partial\bar{t}_k}\right)\right), u_k\right)-\frac{\partial }{\partial\bar{t}_k}h^E(\pi(L_ju_j),u_k)+h^E(D_{t_j}^Eu_j,D_{t_k}^Eu_k)   \\
        &=\int_{\Omega_t}h^F\left(L_j\left(\frac{\partial u_j}{\partial\bar{t}_k}\right),u_k\right)d\lambda-\frac{\partial }{\partial\bar{t}_k}\int_{\Omega_t}h^F(L_ju_j,u_k)d\lambda+h^E(D_{t_j}^Eu_j,D_{t_k}^Eu_k)   \\
        &=\int_{\Omega_t}h^F\left([L_j,\frac{\partial}{\partial\bar{t}_k}]u_j,u_k\right)d\lambda+h^E(D_{t_j}^Eu_j,D_{t_k}^Eu_k)   \\
        &\quad -\int_{\Omega_t}h^F(L_ju_j,L_ku_k)d\lambda-\sum_q\int_ {\Omega_t}h^F\left(\frac{\partial}{\partial\bar{z}_q}(L_ju_j),\nu_k^q u_k\right)d\lambda   \\
        &=\int_{\Omega_t}h^F\left(\Theta_{jk}^{(F,h^F)}u_j,u_k\right)d\lambda-\int_{\Omega_t}h^F(\pi_{\perp}(L_ju_j),\pi_{\perp}(L_ku_k))d\lambda   \\
         &\quad +\int_{\Omega_t}h^F\left([N_j,\frac{\partial}{\partial\bar{t}_k}]u_j,u_k\right)d\lambda-\sum_q\int_{\Omega_t}h^F\left(\frac{\partial}{\partial\bar{z}_q}(L_ju_j),\nu_k^q u_k\right)d\lambda  \\
        &=\int_{\Omega_t}\varphi_{jk}u_{j}\overline{u_{k}}d\lambda-\int_{\Omega_t}h^F(\pi_{\perp}(L_ju_j),\pi_{\perp}(L_ku_k))d\lambda   \\
        &\quad +\int_{\Omega_t}h^F\left([N_j,\frac{\partial}{\partial\bar{t}_k}]u_j,u_k\right)d\lambda-\sum_q\int_{\Omega_t}h^F(\frac{\partial}{\partial\bar{z}_q}(L_ju_j),\nu_k^q u_k)d\lambda.
       \end{align*}
        The proof of Theorem \ref{thm:line bundle} will be completed if we have the following two claims.\\
       {\bf Claim 1:} For the last two terms of the above equality, we have 
       \begin{align*}
        &\quad h^F([N_j,\frac{\partial}{\partial\bar{t}_k}]u_j,u_k)-\sum_qh^F(\frac{\partial}{\partial\bar{z}_q}(L_ju_j),\nu_k^q u_k)\\
        &=\left(\sum_p\nu_j^p\varphi_{p\bar k}+\sum_q\overline{\nu_k^q}\varphi_{j\bar q}+\sum_{p,q}\nu_j^p\overline{\nu_k^q}\varphi_{p\bar q}\right)
        u_j\overline{u_k}e^{-\varphi}\\
        &\quad -\sum_p\frac{\partial}{\partial z_p}\left(\left(\frac{\partial \nu_j^p}{\partial\bar{t}_k}+\sum_q\overline{\nu_k^q}\frac{\partial\nu_j^p}{\partial\bar{z}_q}\right)
        u_j\overline{u_k}e^{-\varphi}\right)
        +\sum_{p,q}\frac{\partial\nu_j^p}{\partial\bar{z}_q}\frac{\partial\overline{\nu_k^q}}{\partial z_p}u_j\overline{u_k}e^{-\varphi}.
        \end{align*}
        {\bf Claim 2:} We moreover have 
        $$
        \sum_p\int_{\Omega_t}\frac{\partial}{\partial z_p}\left(\left(\frac{\partial \nu_j^p}{\partial\bar{t}_k}+\sum_q\overline{\nu_k^q}\frac{\partial\nu_j^p}{\partial\bar{z}_q}\right)u_j\overline{u_k}e^{-\varphi}\right)d\lambda
        =\int_{\partial\Omega_t}H(\rho)_{j\bar k}dS.
        $$
        {\bf The proof of Claim 1:} Note that  
        \begin{align*}
            &\quad [N_j,\frac{\partial}{\partial\bar{t}_k}]u_j\\
            &=\sum_p\nu_j^p\left(\frac{\partial^2 u_j}{\partial \bar{t}_k\partial z_p}-\frac{\partial u_j}{\partial t_k}\varphi_p\right)+\sum_p\frac{\partial\nu_j^p}{\partial z_p}
            \frac{\partial u_j}{\partial \bar{t}_k}\\
            &\quad -\sum_p\frac{\partial}{\partial \bar{t}_k}\left(\nu_j^p\left(\frac{\partial u_j}{\partial z_p}-u_j\varphi_p\right)\right)
            -\sum_p\frac{\partial}{\partial \bar{t}_k}\left(\frac{\partial\nu_j^p}{\partial z_p}u_j\right)\\
            &=\sum_pu_j\frac{\partial(\nu_j^p\varphi_p)}{\partial\bar{t}_k}-\sum_p\frac{\partial\nu_j^p}{\partial\bar{t}_k}\frac{\partial u_j}{\partial z_p}
            -\sum_p\frac{\partial^2\nu_j^p}{\partial\bar{t}_k\partial z_p}u_j,
        \end{align*}
        then we get 
        \begin{equation*}
        h^F([N_j,\frac{\partial}{\partial\bar{t}_k}]u_j,u_k)=
        \sum_p\left(\frac{\partial(\nu_j^p\varphi_p)}{\partial\bar{t}_k}
            -\frac{\partial^2\nu_j^p}{\partial\bar{t}_k\partial z_p}\right)u_j\overline{u_k}e^{-\varphi}
            -\sum_p\frac{\partial\nu_j^p}{\partial\bar{t}_k}\frac{\partial u_j}{\partial z_p}\overline{u_k}e^{-\varphi}.
        \end{equation*}
        We also know 
         \begin{align*}
        \frac{\partial}{\partial\bar{z}_q}(L_ju_j)&=\frac{\partial}{\partial \bar{z}_q}\left(\frac{\partial u_j}{\partial t_j}-u_j\varphi_j+\sum_p \left(\frac{\partial (\nu_j^p u_j)}{\partial z_p}-\nu_j^pu_j\varphi_p\right)\right)\\
        &=-u_j\varphi_{j\bar q}+\sum_p\frac{\partial}{\partial z_p}\left(\frac{\partial\nu_j^p}{\partial\bar{z}_q}u_j\right)-
        \sum_p\frac{\partial(\nu_j^p\varphi_p)}{\partial\bar{z}_q}u_j,
        \end{align*}
        then we have 
        \begin{align*}
         &\quad \sum_q h^F\left(\frac{\partial}{\partial\bar{z}_q}(L_ju_j),\nu_k^q u_k\right)\\
         &=\sum_{p,q}\frac{\partial}{\partial z_p}\left(\frac{\partial\nu_j^p}{\partial\bar{z}_q}u_j\right)\overline{\nu_k^qu_k}e^{-\varphi}
         -\sum_{q}\left(\varphi_{j\bar q}+\sum_p\frac{\partial(\nu_j^p\varphi_p)}{\partial\bar{z}_q}\right)u_j\overline{\nu_k^qu_k}e^{-\varphi}.
        \end{align*}
         Thus, we have
        \begin{align*}
        &\quad h^F([N_j,\frac{\partial}{\partial\bar{t}_k}]u_j,u_k)-\sum_qh^F(\frac{\partial}{\partial\bar{z}_q}(L_ju_j),\nu_k^q u_k) \\
        &=\sum_p\left(\frac{\partial(\nu_j^p\varphi_p)}{\partial\bar{t}_k}
            -\frac{\partial^2\nu_j^p}{\partial\bar{t}_k\partial z_p}-\sum_q\frac{\partial^2\nu_j^p}{\partial z_p\partial\bar{z}_q}
            \overline{\nu_k^q}\right)u_j\overline{u_k}e^{-\varphi} \\
        &\quad +\sum_q\left(\varphi_{j\bar q}+\sum_p\frac{\partial(\nu_j^p\varphi_p)}{\partial\bar{z}_q}\right)
            u_j\overline{\nu_k^qu_k}e^{-\varphi} \\
        &\quad-\sum_p\frac{\partial\nu_j^p}{\partial\bar{t}_k}\frac{\partial u_j}{\partial z_p}\overline{u_k}e^{-\varphi}
        -\sum_{p,q}\frac{\partial\nu_j^p}{\partial\bar{z}_q}\frac{\partial u_j}{\partial z_p}\overline{\nu_k^q u_k}e^{-\varphi} \\
        &=\left(\sum_p\nu_j^p\varphi_{p\bar k}+\sum_q\overline{\nu_k^q}\varphi_{j\bar q}+\sum_{p,q}\nu_j^p\overline{\nu_k^q}\varphi_{p\bar q}\right)
        u_j\overline{u_k}e^{-\varphi} \\
        &\quad +\sum_p\left(\frac{\partial\nu_j^p}{\partial\bar{t}_k}\varphi_p+\sum_q\frac{\partial\nu_j^p}{\partial\bar{z}_q}\varphi_p\overline{\nu_k^q}
        -\frac{\partial^2\nu_j^p}{\partial \bar{t}_k\partial z_p}-\sum_q\frac{\partial^2\nu_j^p}{\partial z_p\partial\bar{z}_q}\overline{\nu_k^q}\right)
        u_j\overline{u_k}e^{-\varphi} \\
        &\quad-\sum_p\frac{\partial\nu_j^p}{\partial\bar{t}_k}\frac{\partial u_j}{\partial z_p}\overline{u_k}e^{-\varphi}
        -\sum_{p,q}\frac{\partial\nu_j^p}{\partial\bar{z}_q}\frac{\partial u_j}{\partial z_p}\overline{\nu_k^q u_k}e^{-\varphi} \\
        &=\left(\sum_p\nu_j^p\varphi_{p\bar k}+\sum_q\overline{\nu_k^q}\varphi_{j\bar q}+\sum_{p,q}\nu_j^p\overline{\nu_k^q}\varphi_{p\bar q}\right)
        u_j\overline{u_k}e^{-\varphi}\\
        &\quad +\sum_p\left(\frac{\partial\nu_j^p}{\partial\bar{t}_k}+\sum_q\frac{\partial\nu_j^p}{\partial\bar{z}_q}\overline{\nu_k^q}\right)
        \frac{\partial\varphi}{\partial z_p}u_j\overline{u_k}e^{-\varphi} \\
        &\quad  -\sum_p\left(\frac{\partial^2\nu_j^p}{\partial \bar{t}_k\partial z_p}+\sum_q\frac{\partial^2\nu_j^p}{\partial z_p\partial\bar{z}_q}\overline{\nu_k^q}
        +\sum_{q}\frac{\partial\nu_j^p}{\partial\bar{z}_q}\frac{\partial\overline{\nu_k^q}}{\partial z_p}\right)
        u_j\overline{u_k}e^{-\varphi} \\
        &\quad-\left(\sum_p\frac{\partial\nu_j^p}{\partial\bar{t}_k}
        +\sum_{p,q}\frac{\partial\nu_j^p}{\partial\bar{z}_q}\overline{\nu_k^q}\right)\frac{\partial u_j}{\partial z_p}\overline{u_k}e^{-\varphi}
        +\sum_{p,q}\frac{\partial\nu_j^p}{\partial\bar{z}_q}\frac{\partial\overline{\nu_k^q}}{\partial z_p}u_j\overline{u_k}e^{-\varphi} \\
        &=\left(\sum_p\nu_j^p\varphi_{p\bar k}+\sum_q\overline{\nu_k^q}\varphi_{j\bar q}+\sum_{p,q}\nu_j^p\overline{\nu_k^q}\varphi_{p\bar q}\right)
        u_j\overline{u_k}e^{-\varphi} \\
        &\quad -\sum_p\frac{\partial}{\partial z_p}\left(\left(\frac{\partial \nu_j^p}{\partial\bar{t}_k}+\sum_q\overline{\nu_k^q}\frac{\partial\nu_j^p}{\partial\bar{z}_q}\right)
        u_j\overline{u_k}e^{-\varphi}\right)
        +\sum_{p,q}\frac{\partial\nu_j^p}{\partial\bar{z}_q}\frac{\partial\overline{\nu_k^q}}{\partial z_p}u_j\overline{u_k}e^{-\varphi}. 
        \end{align*}
        The proof of {\bf Claim 1} is complete.\\
        {\bf The proof of Claim 2:} We know $V_{\bar k}(V_j(\rho))|_{\partial\Omega_t}=0$. In fact,
        \begin{align*}
         &\quad V_{\bar k}(V_j(\rho))|_{\partial\Omega_t}\\
         &=V_{\bar{k}}\left.\left(\rho_j+\sum_p\left(\frac{\rho_j\rho^{\bar p}}{\rho-|\partial\rho|^2}-\sum_q\left(\rho_{j\bar q}\rho^{q\bar p}+\frac{\rho_{j\bar q}\rho^q\rho^{\bar p}}{\rho-|\partial\rho|^2}\right)
         \right)\rho_p\right)\right|_{\partial\Omega_t} \\
        &=V_{\bar{k}}\left.\left(\frac{\rho\rho_j}{\rho-|\partial\rho|^2}-\sum_q\frac{\rho\rho_{j\bar q}\rho^q}{\rho-|\partial\rho|^2}\right)\right|_{\partial\Omega_t} \\
        &=\left(-\frac{\rho_j}{|\partial\rho|^2}+\sum_q\frac{\rho_{j\bar q}\rho^q}{|\partial\rho|^2}\right)V_{\bar{k}}(\rho)|_{\partial\Omega_t} \\
        &=0.
        \end{align*} 
        By the Stokes theorem, we get
        \begin{align*}
        &\quad -\sum_p\int_{\Omega_t}\frac{\partial}{\partial z_p}\left(\left(\frac{\partial \nu_j^p}{\partial\bar{t}_k}+\sum_q\overline{\nu_k^q}\frac{\partial\nu_j^p}{\partial\bar{z}_q}\right)u_j\overline{u_k}e^{-\varphi}\right)d\lambda\\  
        &=-\sum_p\int_{\partial\Omega_t}\left(\frac{\partial \nu_j^p}{\partial\bar{t}_k}+\sum_q\overline{\nu_k^q}\frac{\partial\nu_j^p}{\partial\bar{z}_q}\right)
        \rho_p u_j\overline{u_k}e^{-\varphi}\frac{dS}{|\nabla\rho|}  \\
        &=\int_{\partial\Omega_t}\left(V_{\bar k}(V_j(\rho))-\sum_p\left(\frac{\partial \nu_j^p}{\partial\bar{t}_k}+\sum_q\overline{\nu_k^q}\frac{\partial\nu_j^p}{\partial\bar{z}_q}
        \right)\rho_p\right)u_j\overline{u_k}e^{-\varphi}\frac{dS}{|\nabla\rho|}  \\
        &=\int_{\partial\Omega_t}\left(\rho_{j\bar{k}}+\sum_p\nu_j^p\overline{\rho_{k\bar p}}+\sum_q\overline{\nu_k^q}\rho_{j\bar{q}}+\sum_{p,q}\nu_j^p\overline{\nu_k^q}\rho_{p\bar q}\right)u_j\overline{u_k}e^{-\varphi}\frac{dS}{|\nabla\rho|}. 
        \end{align*} 
        On $\partial\Omega_t$, we have 
        $$
        \nu_j^p=-\frac{\rho_j\rho^{\bar{p}}}{|\partial\rho|^2}-\sum_s\left(\rho_{j\bar s}\rho^{s\bar p}-\frac{\rho_{j\bar s}\rho^{s}\rho^{\bar{p}}}{|\partial\rho|^2}\right),
        $$
        then we know 
        $$
        \sum_p \nu_j^p\overline{\rho_{k\bar p}}=-\sum_p\frac{\rho_j\rho^{\bar{p}}\overline{\rho_{k\bar p}}}{|\partial\rho|^2}
        -\sum_{p,s}\left(\rho_{j\bar s}\rho^{s\bar p}\overline{\rho_{k\bar p}}-\frac{\rho_{j\bar s}\rho^{s}\rho^{\bar{p}}\overline{\rho_{k\bar p}}}{|\partial\rho|^2}\right),
        $$
        $$
        \sum_q\overline{\nu_k^q}\rho_{j\bar{q}}=
        -\sum_q\frac{\rho_{\bar k}\rho^{q}\rho_{j\bar q}}{|\partial\rho|^2}-\sum_a\left(\overline{\rho_{k\bar a}}\overline{\rho^{a\bar q}}\rho_{j\bar{q}}-\frac{\overline{\rho_{k\bar a}}\rho^{\bar a}\rho^{q}\rho_{j\bar q}}{|\partial\rho|^2}\right),
        $$
        and 
        \begin{align*}
        &\quad \sum_{p,q}\nu_j^p\overline{\nu_k^q}\rho_{p\bar q}\\
        &=-\sum_{q}\overline{\nu_k^q}\left(\frac{\rho_j\rho_{\bar{q}}}{|\partial\rho|^2}+\rho_{j\bar q}-\sum_s\frac{\rho_{j\bar s}\rho^{s}\rho_{\bar{q}}}{|\partial\rho|^2}\right) \\
        &=\sum_{q}\left(\frac{\rho_{\bar k}\rho^{q}}{|\partial\rho|^2}+\sum_a\left(\overline{\rho_{k\bar a}}\overline{\rho^{a\bar q}}-\frac{\overline{\rho_{k\bar a}}\rho^{\bar a}\rho^{q}}{|\partial\rho|^2}\right)\right)\left(\frac{\rho_j\rho_{\bar{q}}}{|\partial\rho|^2}+\rho_{j\bar q}-\sum_s\frac{\rho_{j\bar s}\rho^{s}\rho_{\bar{q}}}{|\partial\rho|^2}\right) \\
        &=\frac{\rho_j\rho_{\bar k}}{|\partial\rho|^2}+\sum_q\frac{\rho_{\bar k}\rho^{q}\rho_{j\bar q}}{|\partial\rho|^2}-\sum_s\frac{\rho_{j\bar s}\rho^s\rho_{\bar k}}{|\partial\rho|^2}
        +\sum_a\frac{\rho_j\overline{\rho_{k\bar{a}}}\rho^{\bar{a}}}{|\partial\rho|^2}+\sum_a\overline{\rho_{k\bar a}}\overline{\rho^{a\bar q}}\rho_{j\bar{q}} \\
        &\quad -\sum_{s,a}\frac{\rho_{j\bar s}\rho^s\overline{\rho_{k\bar a}}\rho^{\bar a}}{|\partial\rho|^2}-\sum_a\frac{\rho_j\overline{\rho_{k\bar{a}}}\rho^{\bar a}}{|\partial\rho|^2}
        -\sum_{q,a}\frac{\rho_{j\bar q}\rho^q\overline{\rho_{k\bar a}}\rho^{\bar a}}{|\partial\rho|^2}+\sum_{s,a}\frac{\rho_{j\bar s}\rho^s\overline{\rho_{k\bar a}}\rho^{\bar a}}{|\partial\rho|^2}
         \\
        &=\frac{\rho_j\rho_{\bar k}}{|\partial\rho|^2}+\sum_a\overline{\rho_{k\bar a}}\overline{\rho^{a\bar q}}\rho_{j\bar{q}}-\sum_{s,a}\frac{\rho_{j\bar s}\rho^s\overline{\rho_{k\bar a}}\rho^{\bar a}}{|\partial\rho|^2}
         .
        \end{align*}
        Hence, on $\partial\Omega_t$, we have
        \begin{align*}
        &\quad \rho_{j\bar{k}}+\sum_p\nu_j^p\overline{\rho_{k\bar p}}+\sum_q\overline{\nu_k^q}\rho_{j\bar{q}}+\sum_{p,q}\nu_j^p\overline{\nu_k^q}\rho_{p\bar q}\label{1chern:6}\\ 
        &=\rho_{j\bar{k}}+\frac{\rho_j\rho_{\bar k}}{|\partial\rho|^2}-\sum_p\frac{\rho_j\rho^{\bar{p}}\overline{\rho_{k\bar p}}}{|\partial\rho|^2}
        -\sum_{p,s}\rho_{j\bar s}\rho^{s\bar p}\overline{\rho_{k\bar p}}+\sum_{p,s}\frac{\rho_{j\bar s}\rho^{s}\rho^{\bar{p}}\overline{\rho_{k\bar p}}}{|\partial\rho|^2}
        -\sum_q\frac{\rho_{\bar k}\rho^{q}\rho_{j\bar q}}{|\partial\rho|^2}  \\
        &=\left(\rho_{j\bar{k}}-\sum_{p,s}\rho_{j\bar{s}}\rho^{s\bar{p}}\overline{\rho_{k\bar{p}}}\right)
        +\frac{(\rho_j-\sum_s\rho_{j\bar{s}}\rho^s)(\rho_{\bar{k}}-\overline{\rho_{k\bar{p}}}\rho^{\bar{p}})}{|\partial\rho|^2} . 
        \end{align*}
        The proof of {\bf Claim 2} is complete, and therefore we completes the proof of Theorem \ref{thm:line bundle}.
           \end{proof}
          \section{The proof of Theorem \ref{thm(extend):Bern direct image bd domain}}
            Let us first give a lemma that will be used in the proof of Theorem \ref{thm(extend):Bern direct image bd domain}. Its proof is also given in the appendix \ref{appendix}.
            \begin{lem}\label{lem:strictly positive vector bundle}
            If $F$ is  Nakano strictly positive on $\Omega$, then for any  $u_1,\cdots,u_n\in C^2(U,E)$, we have (locally)
            \begin{align*}
             &\quad \sum_{j,k,\lambda,\mu}\int_{\Omega_t}H(h^F)_{jk\lambda\mu}u_{j\lambda}\overline{u_{k\mu}}d\lambda-\sum_{j,k}\int_{\Omega_t}h^F(\pi_{\perp}(L_ju_j),\pi_{\perp}(L_ku_k))d\lambda\\
                &\quad+\sum_{j,k,p,q}\int_{\Omega_t}\frac{\partial\nu_j^p}{\partial\bar{z}_q}\frac{\partial\overline{\nu_k^q}}{\partial z_p}h^F(u_j,u_k)d\lambda \\
              &\geq  \int_{\Omega_t} \sum_{j,k,\beta,\gamma}(A_{jk\beta\gamma}-\sum_{p,q,\lambda,\mu}B_{pq\lambda\mu}A_{jq\beta\mu}A_{pk\lambda\gamma})u_{j\beta}\overline{u_{k\gamma}} d\lambda
               ,
			\end{align*} 
				where $B$ is uniquely determined by $A$ and satisfies
				$$
                \sum_{q,\beta}A_{pq\lambda\beta}B_{sq\alpha\beta}=\delta_{ps}\delta_{\lambda\alpha}.
                 $$
            \end{lem}
            By the usual approximation technique, we easily have 
            \begin{cor}\label{cor:positive vector bundle}
            If $F$ is just Nakano positive on $\Omega$, then for any  $u_1,\cdots,u_n\in C^2(U,E)$, we have (locally)
            \begin{align*}
             &\quad \sum_{j,k,\lambda,\mu}\int_{\Omega_t}H(h^F)_{jk\lambda\mu}u_{j\lambda}\overline{u_{k\mu}}d\lambda-\sum_{j,k}\int_{\Omega_t}h^F(\pi_{\perp}(L_ju_j),\pi_{\perp}(L_ku_k))d\lambda \\
                &\quad+\sum_{j,k,p,q}\int_{\Omega_t}\frac{\partial\nu_j^p}{\partial\bar{z}_q}\frac{\partial\overline{\nu_k^q}}{\partial z_p}h^F(u_j,u_k)d\lambda \\
            &\geq 0   .
			\end{align*} 
            \end{cor}
			Now we can prove Theorem \ref{thm(extend):Bern direct image bd domain}. For convenience, we restate it here.
		\begin{thm}[= Theorem \ref{thm(extend):Bern direct image bd domain}]\label{proof}
		Let $\Omega\subset U\times\mc^m$ be a family of bounded domains which has a plurisubharmonic defining function over a domain $U\subset\mc^n,$
        and let $(F,h^F)$ be a Hermitian holomorphic trivial vector bundle of finite rank $r$ defined on some neighborhood of $\overline\Omega$.
        Let $\{e_1,\cdots,e_r\}$ be the canonical holomorphic frame of $F$. For any $t\in U$, set
    	$$E_t:=\{u=\sum_{\lambda} u_\lambda dz_1\wedge\cdots\wedge dz_m\otimes e_\lambda|\ u_\lambda \in \mo(\Omega_t)\cap L^2(\Omega_t)\text{ for all }\lambda\}$$
    	with an inner product $h_t$ given by
    	$$h_t^E(u,v):=\int_{\Omega_t} \sum_{\lambda,\mu}u_\lambda(z)\overline{v_\mu(z)}h^F_{(t,z)}(e_\lambda,e_\mu)d\lambda(z),\ \forall u,v \in E_t.$$
        We set $E:=\cup_{t\in U}E_t$, then $(E,h^E)$ is Nakano positive if $(F,h^F)$ is Nakano positive on $\Omega$, and it is strictly positive in the sense of Nakano
        if $(F,h^F)$ is Nakano strictly positive on $\Omega$.
		\end{thm}
		\begin{proof}
			Notation as above and as in Appendix \ref{appendix}. To prove the theorem, we only need to consider the case $(F,h^F)$ is Nakano strictly positive by the usual approximation
            technique, then we can use Lemma \ref{lem(extend):the positivity of curvature estimate}.
            By Theorem \ref{curvature formula 1}, Lemma \ref{lem:strictly positive vector bundle}, Lemma \ref{lem(extend):the positivity of curvature estimate},
            and Corollary \ref{cor:positive definite}, it suffices to prove that 
            $$
            \int_{\Omega_t} \sum_{j,k,\beta,\gamma}(A_{jk\beta\gamma}-\sum_{p,q,\lambda,\mu}B_{pq\lambda\mu}A_{jq\beta\mu}A_{pk\lambda\gamma})u_{j\beta}\overline{u_{k\gamma}} d\lambda>0
            $$
            for any not all zero smooth sections $u_1,\cdots,u_n$ of $E$.\\
		    \indent Let $(W,\langle-,-\rangle)$ be the (complex) inner product space whose orthonormal basis is given by $\{d\bar{z}_p\otimes e_\lambda,\ d\bar{t}_j\otimes e_\mu|\
             \forall p,j,\lambda,\mu\}$, and let (for the simplicity of notations, we omit $dz$ here)
		    $$U:=\operatorname{span}_\mc\{d\bar{z}_p\otimes e_\lambda|\ \forall p,\lambda\},\ V:=\operatorname{span}_\mc\{d\bar{t}_j\otimes e_\lambda|\ \forall j,\lambda\}.$$
		    We claim  the linear map 
		    $$ S:=\left(\begin{array}{cc}
		    	S_1 & S_2\\
		    	S_3 & S_4
		    \end{array}\right)$$
		    is positive define. In fact, we have 
		    \begin{align*}
		    	&\quad S(\sum_{p,\lambda}u_{p\lambda} d\bar{z}_p\otimes e_\lambda+\sum_{j,\lambda}u_{j\lambda}d\bar{t}_j\otimes e_\lambda)\\
		    	&=\sum_{p,q,\lambda,\mu}A_{pq\lambda\mu}u_{p\lambda}d\bar{z}_q\otimes e_\mu+\sum_{j,q,\lambda,\mu}A_{jq\lambda\mu}u_{j\lambda}d\bar{z}_q\otimes e_\mu \\
		    	&\quad +\sum_{p,k,\lambda,\mu}A_{pk\lambda\mu}u_{p\lambda}d\bar{t}_k\otimes e_\mu+\sum_{j,k,\lambda,\mu}A_{jk\lambda\mu}u_{j\lambda}d\bar{t}_k\otimes e_\mu, 
		    \end{align*}
		    and 
		    \begin{align*}
		    	&\quad \langle S(\sum_{p,\lambda}u_{p\lambda} d\bar{z}_p\otimes e_\lambda+\sum_{j,\lambda}u_{j\lambda}d\bar{t}_j\otimes e_\lambda,
		    	\sum_{p,\lambda}u_{p\lambda} d\bar{z}_p\otimes e_\lambda+\sum_{j,\lambda}u_{j\lambda}d\bar{t}_j\otimes e_\lambda\rangle\\		    	
                &=\sum_{p,q,\lambda,\mu}A_{pq\lambda\mu}u_{p\lambda}\overline{u_{q\mu}}+\sum_{j,q,\lambda,\mu}A_{jq\lambda\mu}u_{j\lambda}\overline{u_{q\mu}} \\
                &\quad +\sum_{k,p,\lambda,\mu}A_{pk\lambda\mu}u_{p\lambda}\overline{u_{k\mu}}+\sum_{j,k,\lambda,\mu}A_{jk\lambda\mu}u_{j\lambda}\overline{u_{k\mu}}. 
		    \end{align*}
		    Since  $(F,h^F)$ is Nakano strictly positive, then we know $S$ is positive define, and then we know $S_4-S_3S_1^{-1}S_2$ is 
            positive definite by Lemma \ref{lem(extend):the positivity of curvature estimate}. \\
            \indent However, from 
		    \begin{align*}
		    	&\quad (S_4-S_3S_1^{-1}S_2)(d\bar{t}_j\otimes e_\lambda)\\
		    	&=(S_4-S_3BS_2)(d\bar{t}_j\otimes e_\lambda) \\
		    	&=\sum_{k,\mu}A_{jk\lambda\mu}d\bar{t}_k\otimes e_\mu-\sum_{p,q,k,\mu,\alpha,\beta}A_{jq\lambda\mu}B_{pq\alpha\mu}A_{pk\alpha\beta}d\bar{t}_k\otimes e_\beta,
		    \end{align*}
		    we know 
		    \begin{align*}
		    	&\quad \langle(S_4-S_3S_1^{-1}S_2)(\sum_{j,\lambda}u_{j\lambda}d\bar{t}_j\otimes e_\lambda),\sum_{j,\lambda}u_{j\lambda}d\bar{t}_j\otimes e_\lambda\rangle\\
            	&=\sum_{j,k,\lambda,\mu}A_{jk\lambda\mu}u_{j\lambda}\overline{u_{k\mu}}-\sum_{j,k,p,q,\lambda,\mu,\alpha,\beta}A_{jq\lambda\mu}B_{pq\alpha\mu}A_{pk\alpha\beta}u_{j\lambda}
                  \overline{u_{k\beta}} \\
		    	&=\sum_{j,k,\beta,\gamma}\left(A_{jk\beta\gamma}-\sum_{p,q,\lambda,\mu}B_{pq\lambda\mu}A_{jq\beta\mu}A_{pk\lambda\gamma}\right)u_{j\beta}\overline{u_{k\gamma}},
		    \end{align*}
		    so the positive definiteness of $S_4-S_3S_1^{-1}S_2$ implies what we wanted.
		\end{proof}
            \section{The proof of Theorem \ref{thm:vector bundles of holomorphic functions}}
            Now we can prove Theorem \ref{thm:vector bundles of holomorphic functions}, which is a consequence of Theorem \ref{curvature formula 1}.
            \begin{thm}[= Theorem \ref{thm:vector bundles of holomorphic functions}]\label{proof 2}
         Assumptions (or notations) of $\Omega, F,E$ as in Theorem \ref{thm(extend):Bern direct image bd domain}.
         We moreover assume $\Omega$ is a strictly pseudoconvex family of bounded domains over $U$, then $(E,h^E)$ is strictly positive in the sense of Nakano 
         if $(F,h^F)$ is just Nakano positive on $\Omega$.
           \end{thm}
           \begin{proof}
           Since the curvature operator of $(E,h^E)$ is pointwise defined, we may prove the theorem for any fixed $t_0\in U$.
            For any smooth sections $u_1,\cdots,u_n$ of $E$. By Theorem \ref{curvature formula 1} and Corollary \ref{cor:positive vector bundle}, we know   
            $$
            h^E\left(\Theta_{jk}^{(E,h^E)}u_j,u_k\right)\geq \sum_{j,k}\int_{\partial \Omega_{t_0}}H_0(\rho)_{j\bar{k}}h^F(u_j,u_k)\frac{dS}{|\nabla\rho|},
            $$
            where 
            $$
             H_0(\rho)_{j\bar k}:=\rho_{j\bar{k}}-\sum_{p,q}\rho_{j\bar p}\rho^{p\bar{q}}\rho_{q\bar{k}}.
            $$
            \indent By Corollary \ref{cor:positive definite}, there is a constant $\delta_1>0$ such that 
            $\frac{H_0(\rho)(z)}{|\nabla_z\rho(t_0,z)|}-\delta_1 I$ is positive definite (as matrices) for any $z\in\partial \Omega_{t_0}$,
            where $I$ is the identity matrix.
            Recall for any positive definite matrix $A$, there is a positive matrix $B$, such that $A=B^*B$.
            Thus, as $h$ is also positive definite, there is a constant $\delta_2>0$ such that  
            $$
            \sum_{j,k}\int_{\partial \Omega_{t_0}}H_0(\rho)_{j\bar{k}}h^F(u_j,u_k)\frac{dS}{|\nabla\rho|}\geq \delta_2\sum_j\int_{\partial \Omega_{t_0}}|u_j|^2dS.
            $$ 
            By Lemma \ref{lem:basic inequality for holomorphic functions}, we know there is a constant $\delta_3>0$ such that 
            $$
            \sum_j\int_{\partial \Omega_{t_0}}|u_j|^2dS\geq \delta_3\sum_j\int_{\Omega_{t_0}}|u_j|^2d\lambda,
            $$
            which completes the proof.
           \end{proof}
          \section{The proof of Theorem \ref{thm:flatness criterion}}\label{section:criterion}
          In this section, we will give the proof of Theorem \ref{thm:flatness criterion}. A key lemma will be used is the following lemma, see \cite[Lemma 6.2]{Wang17} for its proof.
          \begin{lem}\label{key lemma}
           Let $p\colon \Omega\subset U\times\mc^m(m\geq 2)\rw U$ be a family of bounded domains which has a plurisubharmonic defining function $\rho$ over a domain $U\subset\mc^n,$
        and let $(F,h^F)$ be a Hermitian holomorphic trivial vector bundle of finite rank $r$ defined on some neighborhood of $\overline\Omega$
        such that $\Theta^{(F,h^F)}\equiv 0$ on $\Omega$ and $h^F(v_1,v_1)-h^F(v,v)\equiv 0$ on $\Omega$, where $v,v_1$ is defined in the proof of Lemma \ref{lem:compute}. Then for each $1\leq j\leq n$, $V_j|_{\partial\Omega}$ has a smooth extension $W_j$ to $\overline\Omega$, that is holomorphic on fibers of $\Omega$
        and such that $p_*(W_j)=\frac{\partial}{\partial t_j}$.
          \end{lem}
          Now we give the proof of Theorem \ref{thm:flatness criterion}.
          \begin{thm}[= Theorem \ref{thm:flatness criterion}]\label{flatness 2}
        Let $\Omega\subset U\times\mc^m(m\geq 2)$ be a family of bounded domains which has a plurisubharmonic defining function $\rho$ over a domain $U\subset\mc^n,$
        and let $(F,h^F)$ be a Hermitian holomorphic trivial vector bundle of finite rank $r$ defined on some neighborhood of $\overline\Omega$
        which is Nakano positive on $\Omega$. Define $E$ and $h^E$ as in Theorem \ref{thm(extend):Bern direct image bd domain},  then the following two statements are equivalent:
        \begin{itemize}
          \item[(1)] $\Theta^{(E,h^E)}\equiv 0$ on $U.$
          \item[(2)] $\Theta^{(F,h^F)}\equiv 0$ on $\Omega$ and $\Omega$ is locally trivial.
        \end{itemize}
        \end{thm}
        \begin{proof}
        Firstly, we consider the case (2) implies (1). By assumption, we may assume $\Omega$ is a product domain,
        i.e. $\Omega=U\times\Omega_{t_0}$ for some $t_0\in U$ (note that here we need $\Phi$ smooth to boundary to identify two direct image bundles).
        Thus, we know  $\rho$ is independent of $t\in U$, and $\nu_j^p\equiv 0$ in $\Omega$ for all $j,p$. 
        Note that $\Theta^{(F,h^F)}\equiv 0$ on $\Omega$, then Theorem \ref{curvature formula 1} says that 
        $$\sum_{j,k}h^E(\Theta_{jk}^{(E,h^E)}u_j,u_k)=-\sum_{j,k}\int_{\Omega_t}h^F(\pi_{\perp}(D_{t_j}^Fu_j),\pi_{\perp}(D_{t_k}^Fu_k))d\lambda$$
        for any $u_1,\cdots,u_n\in C^2(U,E)$. Then the computation of the proof of Lemma \ref{lem:compute}
        shows that 
        $$v_1=v_2=0,\ \bar\partial_z v=\sum_s dz_s\wedge D_{z_s}^Fv_1+ v_2=0,$$
        where $v,v_1,v_2$ is defined in the proof of Lemma \ref{lem:compute}. Note also that $v$ is the minimal solution, then we know $v\equiv 0$ in $\Omega_t$ for each $t\in U$. Thus, we conclude that $\Theta^{(E,h^E)}\equiv 0$ on $\Omega$.\\
        \indent Now we prove (1) implies (2). By Corollary \ref{cor:positive vector bundle},
        for any not all zero sections $u_1,\cdots,u_n\in C^2(U,E)$, we have 
        \begin{align*}
        &\sum_{j,k,\lambda,\mu}\int_{\Omega_t}H(h^F)_{jk\lambda\mu}u_{j\lambda}\overline{u_{k\mu}}d\lambda-\sum_{j,k}\int_{\Omega_t}h^F(\pi_{\perp}(L_ju_j),\pi_{\perp}(L_ku_k))d\lambda \\
        &+ \sum_{j,k,p,q}\int_{\Omega_t}\frac{\partial\nu_j^p}{\partial\bar{z}_q}\frac{\partial\overline{\nu_k^q}}{\partial z_p} h^F\left(u_j,u_k\right)d\lambda\geq 0,\ \forall\ t\in U,
       \end{align*}
       where the inequality is strict by Lemma \ref{lem:strictly positive vector bundle}
       and the proof of Theorem \ref{proof} if $\Theta^{(F,h^F)}(z)>0$ for some $z\in \Omega_t$. 
       Using  Corollary \ref{cor:positive definite} and the proof Theorem \ref{proof 2}, 
        for any not all zero sections $u_1,\cdots,u_n\in C^2(U,E)$, we have
        $$\sum_{j,k}\int_{\partial\Omega_t}H_0(\rho)_{j\bar{k}}h^F(u_j,u_k)\frac{dS}{|\nabla\rho|}\geq 0,\ \forall\ t\in U,$$
        where the inequality is strict if some eigenvalues of the matrix $(H_0(\rho)_{j\bar k}(z))$ are positive for some 
        $z\in \partial\Omega_t$. For any not all zero sections $u_1,\cdots,u_n\in C^2(U,E)$, note also that 
        $$\sum_{j,k}\int_{\partial\Omega_t}\frac{(\rho_j-\sum_s\rho_{j\bar{s}}\rho^s)(\rho_{\bar{k}}-\sum_p\overline{\rho_{k\bar{p}}}\rho^{\bar{p}})}{|\partial\rho|^2}h^F(u_j,u_k)
        \frac{dS}{|\nabla\rho|}\geq 0,\ \forall\ t\in U,$$
        where the inequality is strict if $u_j(z)\rho_j(z)-\sum_su_j(z)\rho_{j\bar{s}}(z)\rho^s(z)\neq 0$ for some 
        $j$ and some $z\in\partial\Omega_t$. Thus, by assumption and Theorem \ref{curvature formula 1}, we know $\Theta^{(F,h^F)}\equiv 0$ on $\Omega$,
        and
        $$H(\rho)_{j\bar{k}}\equiv 0\text{ on }\bing_{t\in U}\partial\Omega_t,\ \forall\ j,k.$$
        Moreover, we have $h^F(v_1,v_1)-h^F(v,v)=0\text{ on }\Omega,$ where $v,v_1$ is defined in the proof of Lemma \ref{lem:compute}.\\
        \indent For each $j$, let $W_j$ be the corresponding smooth extension of $V_j|_{\partial\Omega}$ as in Lemma \ref{key lemma}, then for any $j,k$, we have 
        $$[W_j,\overline{W_k}]=\frac{\partial \overline{W_k}}{\partial t_j}-\frac{\partial W_j}{\partial \bar{t}_k}.$$
        Note that $[W_j,\overline{W_k}]|_{\partial\Omega}=[V_j,V_k]|_{\partial\Omega}=0$ for any $j,k$, then we know 
        $$\left.\frac{\partial W_j}{\partial \bar{t}_k}\right|_{\partial\Omega_t}\equiv 0\text{ on } U.$$
        Since each $W_j$ is holomorphic on fibers, then we know each $W_j$ is holomorphic with respect to $t$  by Maximum Modulus Principle. On the other hand, since $[V_j,V_k]\equiv 0$ for any $j,k$, then we have  $[W_j,W_k]|_{\partial\Omega}\equiv 0,$ and then we get $[W_j,W_k]\equiv 0$ in $\Omega$. From this, we will 
        prove $\Omega$ is locally trivial. We may assume $0\in U$, and we only need to prove $\Omega$ is locally trivial near $0$.  We divide it into two cases.\\
        {\bf Case 1:} $n=1$.\\
        \indent   Write  
        $$X:=W_1,\ Y:=X+\overline{X},\ Z:=\sqrt{-1}(X-\overline{X}),$$
        then $Y,Z$ satisfy the following 
        \begin{itemize}
          \item[(1)] $Y,Z$ are tangent to $\partial\Omega;$
          \item[(2)] $[Y,Z]=0$ in $\Omega;$
          \item[(3)] $p_*(Y)=\frac{\partial}{\partial\mu},\ p_*(Z)=\frac{\partial}{\partial\nu}.$
        \end{itemize}
        Choose an open neighborhood $V$ of $\overline\Omega$ such that $X$ is smooth on $V$.
        Let $\alpha,\beta$ be the smooth flows generated by $Y,Z$ respectively on $V$,
        then $\alpha$ commutes with $\beta$ as $[Y,Z]=0$.  Choose $\tau>0$ such that $[-\tau,\tau]\times\overline{\Omega}$ is in the domains of $\alpha$ and $\beta$.
        For any $t\in U$, we write $t:=\mu+\sqrt{-1}\nu,\ \mu,\nu\in\mr$, and we also identify 
        $(-\tau,\tau)^2$ with the domain $(-\tau,\tau)+\sqrt{-1}(-\tau,\tau)$ in $\mc$. Since $Z$ is tangent to $\partial\Omega$, then $\beta(\nu,z)\in\overline\Omega$
        for any $-\tau\leq\nu\leq\tau$ and any $z\in\overline\Omega$. Define a map 
        $$\Phi\colon (-\tau,\tau)^2\times\overline{\Omega_0}\rw V,\ (\mu+\sqrt{-1}\nu,z)\mapsto \alpha(\mu,\beta(\nu,z)),$$
        then we know $\Phi$ restricts to a diffeomorphism (still denoted by $\Phi$) between $(-\tau,\tau)^2\times\Omega_0$ and $ p^{-1}((-\tau,\tau)^2),$
        and
        $$\Phi(\{t\}\times\Omega_0)=\Omega_t,\ \forall\ t\in (-\tau,\tau)^2.$$
        To prove $\Omega$ is locally trivial near $0$, it suffices to prove the following two claims.\\
        {\bf Claim 1:} $\Phi$ is holomorphic on $(-\tau,\tau)^2\times\Omega_0$. \\
        \indent Fix $z_0\in \Omega_0$, and consider the map 
        $$\varphi\colon (-\tau,\tau)^2\rw \Omega,\ t=\mu+\sqrt{-1}\nu\mapsto \alpha(\mu,\beta(\nu,z_0)),$$
        then we get 
        $$\frac{\partial\varphi(t)}{\partial t}=X(\varphi(t))\text{ on }(-\tau,\tau)^2,\ \varphi(0)=z_0,$$
        which implies $\Phi$ is the holomorphic flow generated by $X$, so $\Phi$ is holomorphic on $(-\tau,\tau)^2\times\Omega_0.$\\
        {\bf Claim 2:} $\Phi$ can be extended smoothly to a neighborhood of $V_0:=(-\eta,\eta)^2\times \overline\Omega_0$
        for some $0<\eta<\tau$.\\
        \indent For any $(t,z)\in(-\tau,\tau)^2\times\overline{\Omega_0}$, write $\Phi(t,z)=(t,f_1(t,z),\cdots,f_m(t,z))$.
        By Whitney extension theorem, there is a neighborhood $V_1$ of $(-\tau,\tau)^2\times\overline{\Omega_0}$
        such that, for any $1\leq i\leq m$,  $f_i$ can be extended to a smooth function $g_i$ on $V_1$.
        Choose $0<\eta<\tau$ and choose an open neighborhood $V_2$ of $\overline\Omega_0$ such that $(-\eta,\eta)^2\times V_2\subset V_1.$
        Define a smooth map
        $$\widetilde{\Phi}\colon (-\eta,\eta)^2\times V_2\rw \mc^{m+1},\ (t,z)\mapsto (t,g_1(t,z),\cdots,g_m(t,z)),$$
        then $\widetilde{\Phi}$ is a desired smooth extension of $\Phi$. (Moreover, using the inverse function theorem and the fact $\widetilde{\Phi}|_{(-\eta,\eta)^2\times\overline{\Omega_0}}$ is a diffeomorphism, we know $\widetilde{\Phi}$ is in fact an embedding by shrinking $\eta$ and $V_2$ if necessary).\\  
        {\bf Case 2:} $n\geq 2.$\\
        \indent For any $j$, set 
        $$Y_j:=W_j+\overline{W_j},\ Z_j:=\sqrt{-1}(W_j-\overline{W_j}),$$
        then $Y_j,Z_j$ satisfy the following 
        \begin{itemize}
          \item[(1)] $Y_j,Z_j$ are tangent to $\partial\Omega$ for any $j;$
          \item[(2)] $[Y_j,Y_k]=[Y_j,Z_k]=[Z_j,Z_k]=0$ in $\Omega$ for any $j,k;$
          \item[(3)] $p_*(Y_j)=\frac{\partial}{\partial\mu_j},\ p_*(Z_j)=\frac{\partial}{\partial\nu_j}$ for any $j$, where 
          $$t_j:=\mu_j+\sqrt{-1}\nu_j,\ \mu_j,\nu_j\in\mr.$$
        \end{itemize} 
        The remaining is similar to {\bf Case 1}, so we omit it.\\
        \indent  The proof is complete.
        \end{proof}

\newpage
		\section*{Appendix}\label{appendix}
	  In this section, we give several long computations. All notations are the same with Section \ref{section:curvature formula}.
      Firstly, we give the proof of Theorem \ref{curvature formula 1} when $F$ is a vector bundle.
		 \begin{thm}[= Theorem \ref{curvature formula 1}]\label{curvature_formula}
         For any  $u_1,\cdots,u_n\in C^2(U,E)$, we have (locally) 
        \begin{align*}
        &\quad h^E\left(\Theta_{jk}^{(E,h^E)}u_j,u_k\right)\\
        &=\sum_{\lambda,\mu}\int_{\Omega_t}H(h^F)_{jk\lambda\mu}u_{j\lambda}\overline{u_{k\mu}}d\lambda-\int_{\Omega_t}h^F(\pi_{\perp}(L_ju_j),\pi_{\perp}(L_ku_k))d\lambda \\
        &\quad+\sum_{p,q}\int_{\Omega_t}\frac{\partial\nu_j^p}{\partial\bar{z}_q}\frac{\partial\overline{\nu_k^q}}{\partial z_p} h^F\left(u_j,u_k\right)d\lambda+\int_{\partial\Omega_t}H(\rho)_{j\bar{k}}h^F(u_j,u_k)\frac{dS}{|\nabla\rho|} ,
       \end{align*}
       where 
       $$
        L_ju_j:=D_{t_j}^Fu_j+\sum_p D_{z_p}^F(\nu_j^p u_j),
       $$
       $$
         H(h^F)_{jk\lambda\mu}:=A_{jk\lambda\mu}+\sum_p\nu_j^pA_{pk\lambda\mu}+\sum_q\overline{\nu_k^q}A_{jq\lambda\mu}+\sum_{p,q}\nu_j^p\overline{\nu_k^q}A_{pq\lambda\mu},
       $$
       $$
        H(\rho)_{j\bar{k}}:=\left(\rho_{j\bar{k}}-\sum_{p,s}\rho_{j\bar{s}}\rho^{s\bar{p}}\overline{\rho_{k\bar{p}}}\right)
        +\frac{(\rho_j-\sum_s\rho_{j\bar{s}}\rho^s)(\rho_{\bar{k}}-\sum_p\overline{\rho_{k\bar{p}}}\rho^{\bar{p}})}{|\partial\rho|^2}.
       $$
       In particular, the curvature operator $\Theta_{jk}^{(E,h^E)}$ is pointwise defined, i.e. it does not  involve derivatives with respect to $t\in U$.
        \end{thm}
        \begin{proof}
        Fix $t\in U,\ j,k$. For any $u\in C^1(U,F)$, set  
        $$
         N_j u:=D_{z_p}^F(\nu_j^p u)=\sum_p\nu_j^p D_{z_p}^Fu+\sum_p\frac{\partial \nu_{j}^p}{\partial z_p}u.
        $$
       Similar to the proof of Theorem \ref{thm:line bundle}, we have 
         \begin{align*}
       &\quad h^E\left(\Theta_{jk}^{(E,h^E)}u_j,u_k\right)\\
        &=\int_{\Omega_t}h^F\left(\Theta_{jk}^{(F,h^F)}u_j,u_k\right)d\lambda-\int_{\Omega_t}h^F(\pi_{\perp}(L_ju_j),\pi_{\perp}(L_ku_k))d\lambda   \\
         &\quad +\int_{\Omega_t}h^F\left([N_j,\frac{\partial}{\partial\bar{t}_k}]u_j,u_k\right)d\lambda-\sum_q\int_{\Omega_t}h^F\left(\frac{\partial}{\partial\bar{z}_q}(L_ju_j),\nu_k^q u_k\right)d\lambda  \\
        &=\sum_{\lambda,\mu}\int_{\Omega_t}A_{jk\lambda\mu}u_{j\lambda}\overline{u_{k\mu}}d\lambda-\int_{\Omega_t}h^F(\pi_{\perp}(L_ju_j),\pi_{\perp}(L_ku_k))d\lambda   \\
        &\quad +\int_{\Omega_t}h^F\left([N_j,\frac{\partial}{\partial\bar{t}_k}]u_j,u_k\right)d\lambda-\sum_q\int_{\Omega_t}h^F(\frac{\partial}{\partial\bar{z}_q}(L_ju_j),\nu_k^q u_k)d\lambda .
       \end{align*}
       Now we only need to prove the following two claims.\\
       {\bf Claim 1:} For the last two terms of the above equality, we have 
       \begin{align*}
        &\quad h^F([N_j,\frac{\partial}{\partial\bar{t}_k}]u_j,u_k)-\sum_qh^F(\frac{\partial}{\partial\bar{z}_q}(L_ju_j),\nu_k^q u_k)\\
        &=\sum_{\lambda,\mu}\left(\sum_p\nu_j^pA_{pk\lambda\mu}+\sum_q\overline{\nu_k^q}A_{jq\lambda\mu}+\sum_{p,q}\nu_j^p\overline{\nu_k^q}A_{pq\lambda\mu}\right)u_{j\lambda}\overline{u_{k\mu}}
          \\
        &\quad -\frac{\partial}{\partial z_p}\left(\left(\frac{\partial \nu_j^p}{\partial\bar{t}_k}+\sum_q\overline{\nu_k^q}\frac{\partial\nu_j^p}{\partial\bar{z}_q}\right)h^F(u_j,u_k)\right)+\sum_{p,q}\frac{\partial\nu_j^p}{\partial\bar{z}_q}\frac{\partial\overline{\nu_k^q}}{\partial z_p}h^F(u_j,u_k).
        \end{align*}
        {\bf Claim 2:} We moreover have 
        $$
        \sum_p\int_{\Omega_t}\frac{\partial}{\partial z_p}\left(\left(\frac{\partial \nu_j^p}{\partial\bar{t}_k}+\sum_q\overline{\nu_k^q}\frac{\partial\nu_j^p}{\partial\bar{z}_q}\right)h^F(u_j,u_k)\right)d\lambda
        =\int_{\partial\Omega_t}H(\rho)_{j\bar k}h^F(u_j,u_k)dS.
        $$
        \indent The proof of {\bf Claim 2} is almost the same with the proof of {\bf Claim 2} in 
        Theorem \ref{thm:line bundle}, so we only prove {\bf Claim 1}.\\ 
        \indent Note that  
        \begin{align*}
            &\quad [N_j,\frac{\partial}{\partial\bar{t}_k}]u_j\\
            &=\sum_p\nu_j^p D_{z_p}^F\left(\frac{\partial u_j}{\partial\bar{t}_k}\right)+\sum_p\frac{\partial \nu_j^p}{\partial z_p}\frac{\partial u_j}{\partial\bar{t}_k}
            -\sum_p\frac{\partial }{\partial\bar{t}_k}\left(\nu_j^p D_{z_p}^Fu_j+\frac{\partial\nu_j^p}{\partial z_p}u_j\right) \\
            &=\sum_{p,\lambda}\nu_j^p\left(\frac{\partial^2 u_{j\lambda}}{\partial z_p\partial\bar{t}_k}dz\otimes e_\lambda +\frac{\partial u_{j\lambda}}{\partial\bar{t}_k}dz\otimes D_{z_p}^Fe_\lambda\right) \\
            &\quad -\sum_p\frac{\partial \nu_j^p}{\partial\bar{t}_k} D_{z_p}^Fu_j-\sum_p\nu_j^p\frac{\partial}{\partial\bar{t}_k}(D_{z_p}^Fu_j)
            -\sum_p\frac{\partial^2\nu_{j}^p}{\partial z_p\partial\bar{t}_k}u_j \\
            &=\sum_{p,\lambda}\nu_j^p\left(\frac{\partial^2 u_{j\lambda}}{\partial z_p\partial\bar{t}_k}dz\otimes e_\lambda +\frac{\partial u_{j\lambda}}{\partial\bar{t}_k}dz\otimes D_{z_p}^Fe_\lambda\right) \\
            &\quad -\sum_{p,\lambda}\frac{\partial \nu_j^p}{\partial\bar{t}_k} \left(\frac{\partial u_{j\lambda}}{\partial z_p}dz\otimes e_\lambda+ u_{j\lambda}dz\otimes D_{z_p}^Fe_\lambda\right) \\
            &\quad -\sum_{p,\lambda}\nu_j^p\frac{\partial}{\partial\bar{t}_k}\left(\frac{\partial u_{j\lambda}}{\partial z_p}dz\otimes e_\lambda+ u_{j\lambda}dz\otimes D_{z_p}^Fe_\lambda\right)
            -\frac{\partial^2\nu_{j}^p}{\partial z_p\partial\bar{t}_k}u_j \\
            &=-\sum_{p,\lambda}\frac{\partial \nu_j^p}{\partial\bar{t}_k} \left(\frac{\partial u_{j\lambda}}{\partial z_p}dz\otimes e_\lambda+ u_{j\lambda}dz\otimes D_{z_p}^Fe_\lambda\right) \\
            &\quad -\sum_{p,\lambda}\nu_j^pu_{j\lambda}dz\otimes \frac{\partial}{\partial\bar{t}_k}(D_{z_p}^Fe_\lambda)-\sum_p\frac{\partial^2\nu_{j}^p}{\partial z_p\partial\bar{t}_k}u_j \\
            &=-\sum_{p,\alpha}\frac{\partial \nu_j^p}{\partial\bar{t}_k} \left(\frac{\partial u_{j\alpha}}{\partial z_p}+\sum_{\lambda}u_{j\lambda}\Gamma_{p\lambda}^\alpha\right)
            dz\otimes e_\alpha \\
            &\quad -\sum_{p,\alpha}\left(\sum_{\lambda}\nu_j^pu_{j\lambda}\frac{\partial \Gamma_{p\lambda}^\alpha}{\partial\bar{t}_k}+\frac{\partial^2\nu_{j}^p}{\partial z_p\partial\bar{t}_k}u_{j\alpha}\right)dz\otimes e_\alpha,
        \end{align*}
        then we get 
        \begin{align*}
        &\quad h^F([N_j,\frac{\partial}{\partial\bar{t}_k}]u_j,u_k)\\ 
        &=-\sum_{p,\alpha,\beta}\frac{\partial \nu_j^p}{\partial\bar{t}_k} \left(\frac{\partial u_{j\alpha}}{\partial z_p}+\sum_{\lambda}u_{j\lambda}\Gamma_{p\lambda}^\alpha\right)
         \overline{u_{k\beta}}h_{\alpha\bar\beta}^F   \\
        &\quad -\sum_{p,\alpha,\beta}\left(\sum_{\lambda}\nu_j^pu_{j\lambda}\frac{\partial \Gamma_{p\lambda}^\alpha}{\partial\bar{t}_k}+\frac{\partial^2\nu_{j}^p}{\partial z_p\partial\bar{t}_k}u_{j\alpha}\right)\overline{u_{k\beta}}h_{\alpha\bar\beta}^F  \\
        &=-\sum_{p,\lambda,\mu}\frac{\partial \nu_j^p}{\partial\bar{t}_k}\frac{\partial u_{j\lambda}}{\partial z_p}
         \overline{u_{k\mu}}h_{\lambda\bar\mu}^F-\sum_{p,\lambda,\mu}\frac{\partial \nu_j^p}{\partial\bar{t}_k}u_{j\lambda}\overline{u_{k\mu}}\frac{\partial h_{\lambda\bar\mu}^F}{\partial z_p} \\
        &\quad +\sum_{p,\lambda,\mu}\nu_j^pA_{pk\lambda\mu}u_{j\lambda}\overline{u_{k\mu}}-\sum_{p,\lambda,\mu}\frac{\partial^2\nu_{j}^p}{\partial z_p\partial\bar{t}_k}u_{j\lambda}\overline{u_{k\mu}}h_{\lambda\bar\mu}^F . 
        \end{align*}
        We also know 
         \begin{align*}
        &\quad \frac{\partial}{\partial\bar{z}_q}(L_ju_j)\\
        &=\sum_{\lambda}\frac{\partial}{\partial\bar{z}_q}\left(\frac{\partial u_{j\lambda}}{\partial t_j}
         +\sum_p\nu_j^p\frac{\partial u_{j\lambda}}{\partial z_p}+\sum_p\frac{\partial\nu_j^p}{\partial z_p}u_{j\lambda}\right)dz\otimes e_\lambda \\
        &\quad+\sum_{\lambda,\mu}\frac{\partial}{\partial\bar{z}_q}\left(u_{j\lambda}\Gamma_{j\lambda}^\mu+\sum_p\nu_j^pu_{j\lambda}\Gamma_{p\lambda}^\mu\right) dz\otimes e_\mu \\
        &=\sum_{p,\lambda}\left(\frac{\partial \nu_j^p}{\partial\bar{z}_q}\frac{\partial u_{j\lambda}}{\partial z_p}+\frac{\partial^2\nu_j^p}{\partial z_p\partial\bar{z}_q}u_{j\lambda}\right)
        dz\otimes e_\lambda \\
        &\quad -\sum_{\lambda,\mu}u_{j\lambda}\left(A_{jq\lambda}^\mu-\sum_p\frac{\partial\nu_j^p}{\partial\bar{z}_q}\Gamma_{p\lambda}^\mu
        +\sum_p\nu_{j}^pA_{pq\lambda}^\mu\right)dz\otimes e_\mu ,
        \end{align*}
        then we have 
        \begin{align*}
         &\quad \sum_q h^F\left(\frac{\partial}{\partial\bar{z}_q}(L_ju_j),\nu_k^q u_k\right)\\
         &=\sum_{p,q,\lambda,\mu}\left(\frac{\partial \nu_j^p}{\partial\bar{z}_q}\frac{\partial u_{j\lambda}}{\partial z_p}+\frac{\partial^2\nu_j^p}{\partial z_p\partial\bar{z}_q}u_{j\lambda}\right)
         \overline{\nu_k^q}\overline{u_{k\mu}}h_{\lambda\bar{\mu}}^F \\
         &\quad -\sum_{q,\lambda,\mu,\alpha}u_{j\lambda}\left(A_{jq\lambda}^\alpha-\sum_p\frac{\partial\nu_j^p}{\partial\bar{z}_q}\Gamma_{p\lambda}^\alpha
        +\sum_p\nu_{j}^pA_{pq\lambda}^\alpha\right) \overline{\nu_k^q}\overline{u_{k\mu}}h_{\alpha\bar\mu}^F \\
        &=\sum_{p,q,\lambda,\mu}\left(\frac{\partial \nu_j^p}{\partial\bar{z}_q}\frac{\partial u_{j\lambda}}{\partial z_p}+\frac{\partial^2\nu_j^p}{\partial z_p\partial\bar{z}_q}u_{j\lambda}\right)
         \overline{\nu_k^q}\overline{u_{k\mu}}h_{\lambda\bar{\mu}}^F-\sum_{q,\lambda,\mu}\overline{\nu_k^q}A_{jq\lambda\mu}u_{j\lambda}\overline{u_{k\mu}} \\
         &\quad +\sum_{p,q,\lambda,\mu} \overline{\nu_k^q}\frac{\partial\nu_j^p}{\partial\bar{z}_q}
        u_{j\lambda}\overline{u_{k\mu}}\frac{\partial h_{\lambda\bar{\mu}}^F}{\partial z_p}-\sum_p\nu_{j}^p\overline{\nu_k^q}u_{j\lambda}\overline{u_{k\mu}}A_{pq\lambda\mu} .
        \end{align*}
         Thus, we have
        \begin{align*}
        &\quad h^F([N_j,\frac{\partial}{\partial\bar{t}_k}]u_j,u_k)-\sum_qh^F(\frac{\partial}{\partial\bar{z}_q}(L_ju_j),\nu_k^q u_k) \\
        &=-\sum_{p,\lambda,\mu}\frac{\partial \nu_j^p}{\partial\bar{t}_k}\frac{\partial u_{j\lambda}}{\partial z_p}
         \overline{u_{k\mu}}h_{\lambda\bar\mu}^F-\sum_{p,\lambda,\mu}\frac{\partial \nu_j^p}{\partial\bar{t}_k}u_{j\lambda}\overline{u_{k\mu}}\frac{\partial h_{\lambda\bar\mu}^F}{\partial z_p}  \\
        &\quad +\sum_{p,\lambda,\mu}\nu_j^pA_{pk\lambda\mu}u_{j\lambda}\overline{u_{k\mu}}-\sum_{p,\lambda,\mu}\frac{\partial^2\nu_{j}^p}{\partial z_p\partial\bar{t}_k}u_{j\lambda}\overline{u_{k\mu}}h_{\lambda\bar\mu}^F  \\
        &\quad -\sum_{p,q,\lambda,\mu}\left(\frac{\partial \nu_j^p}{\partial\bar{z}_q}\frac{\partial u_{j\lambda}}{\partial z_p}+\frac{\partial^2\nu_j^p}{\partial z_p\partial\bar{z}_q}u_{j\lambda}\right)
         \overline{\nu_k^q}\overline{u_{k\mu}}h_{\lambda\bar{\mu}}^F+\sum_{q,\lambda,\mu}\overline{\nu_k^q}A_{jq\lambda\mu}u_{j\lambda}\overline{u_{k\mu}}  \\
         &\quad -\sum_{p,q,\lambda,\mu} \overline{\nu_k^q}\frac{\partial\nu_j^p}{\partial\bar{z}_q}
        u_{j\lambda}\overline{u_{k\mu}}\frac{\partial h_{\lambda\bar{\mu}}^F}{\partial z_p}+\sum_p\nu_{j}^p\overline{\nu_k^q}u_{j\lambda}\overline{u_{k\mu}}A_{pq\lambda\mu} \\
        &=\sum_{\lambda,\mu}\left(\sum_p\nu_j^pA_{pk\lambda\mu}+\sum_q\overline{\nu_k^q}A_{jq\lambda\mu}+\sum_{p,q}\nu_j^p\overline{\nu_k^q}A_{pq\lambda\mu}\right)u_{j\lambda}\overline{u_{k\mu}}
        \\ 
        &\quad -\sum_{p,\lambda,\mu}\left(\frac{\partial \nu_j^p}{\partial\bar{t}_k}+\sum_q\frac{\partial \nu_j^p}{\partial\bar{z}_q}\overline{\nu_k^q}\right)\frac{\partial u_{j\lambda}}{\partial z_p}
         \overline{u_{k\mu}}h_{\lambda\bar\mu}^F  \\
         &\quad-\sum_{p,\lambda,\mu}\left(\frac{\partial \nu_j^p}{\partial\bar{t}_k}
         +\sum_q\overline{\nu_k^q}\frac{\partial\nu_j^p}{\partial\bar{z}_q}\right)u_{j\lambda}\overline{u_{k\mu}}\frac{\partial h_{\lambda\bar\mu}^F}{\partial z_p}  \\
         &\quad -\sum_{p,\lambda,\mu}\left(\frac{\partial^2\nu_{j}^p}{\partial z_p\partial\bar{t}_k}+\sum_q
        \frac{\partial^2\nu_j^p}{\partial z_p\partial\bar{z}_q}\overline{\nu_k^q}+\sum_q\frac{\partial\nu_j^p}{\partial\bar{z}_q}\frac{\partial\overline{\nu_k^q}}{\partial z_p}\right)u_{j\lambda}\overline{u_{k\mu}}h_{\lambda\bar\mu}^F  \\
        &\quad +\sum_{p,q,\lambda,\mu}\frac{\partial\nu_j^p}{\partial\bar{z}_q}\frac{\partial\overline{\nu_k^q}}{\partial z_p}u_{j\lambda}\overline{u_{k\mu}}h_{\lambda\bar\mu}^F \\
        &=\sum_{\lambda,\mu}\left(\sum_p\nu_j^pA_{pk\lambda\mu}+\sum_q\overline{\nu_k^q}A_{jq\lambda\mu}+\sum_{p,q}\nu_j^p\overline{\nu_k^q}A_{pq\lambda\mu}\right)u_{j\lambda}\overline{u_{k\mu}}
          \\
        &\quad -\sum_{\lambda,\mu}\frac{\partial}{\partial z_p}\left(\left(\frac{\partial \nu_j^p}{\partial\bar{t}_k}+\sum_q\overline{\nu_k^q}\frac{\partial\nu_j^p}{\partial\bar{z}_q}\right)u_{j\lambda}\overline{u_{k\mu}}h_{\lambda\bar\mu}^F\right)  \\
        &\quad +\sum_{p,q,\lambda,\mu}\frac{\partial\nu_j^p}{\partial\bar{z}_q}\frac{\partial\overline{\nu_k^q}}{\partial z_p}u_{j\lambda}\overline{u_{k\mu}}h_{\lambda\bar\mu}^F \\
        &=\sum_{\lambda,\mu}\left(\sum_p\nu_j^pA_{pk\lambda\mu}+\sum_q\overline{\nu_k^q}A_{jq\lambda\mu}+\sum_{p,q}\nu_j^p\overline{\nu_k^q}A_{pq\lambda\mu}\right)u_{j\lambda}\overline{u_{k\mu}}
          \\
        &\quad -\frac{\partial}{\partial z_p}\left(\left(\frac{\partial \nu_j^p}{\partial\bar{t}_k}+\sum_q\overline{\nu_k^q}\frac{\partial\nu_j^p}{\partial\bar{z}_q}\right)h^F(u_j,u_k)\right)+\sum_{p,q}\frac{\partial\nu_j^p}{\partial\bar{z}_q}\frac{\partial\overline{\nu_k^q}}{\partial z_p}h^F(u_j,u_k) . 
        \end{align*}
        The proof of {\bf Claim 1} is complete, and then we have completed the proof of Theorem \ref{curvature_formula}.
        \end{proof}
        We  also give an elementary lemma of linear algebra. Although this lemma is simple and well known, we contains its proof as it is important in the proof of Theorem 
           \ref{thm(extend):Bern direct image bd domain}. For this, we recall some definitions.\\
            \indent Let $(W,\langle-,-\rangle)$ (the inner product is complex linear with respect to the first entry) be a finite dimensional (complex) inner product space, and let $T\colon W\rw W$ be a linear map, then $T$ is positive definite (resp. positive semi-definite) if and only if $\langle Tx,x\rangle> 0$
            (resp. $\langle Tx,x\rangle\geq 0$) for any $x\in W\setminus\{0\}.$ \\
            \indent Let $T\colon W_1\rw W_2$ be a linear map between two finite dimensional inner product spaces, then the adjoint map $T^*$ of $T$ is defined by the following formula
            $$\langle Tx,y\rangle=\langle x,T^*y\rangle,\ \forall x\in W_1,\ y\in W_2.$$
			\begin{lem}\label{lem(extend):the positivity of curvature estimate}
				Let $(W,\langle-,-\rangle)$ be a finite dimensional (complex) inner product space, $U$ be a linear subspace of $W$, and let $V$ be the orthogonal complement of 
                 $U$ in $W$. Suppose we are given four linear maps 
				$$T_1\colon U\rw U,\ T_2\colon V\rw U,\ T_3\colon U\rw V,\ T_4\colon V\rw V$$ 
				such that 
               \begin{itemize}
				\item[(1)] $T_1$ is invertible;
				\item[(2)] The linear map $T:=\left(\begin{array}{cc}
					T_1 & T_2\\
					T_3 & T_4
				\end{array}\right)$ (i.e. $T(u+v)=T_1u+T_2v+T_3u+T_4v$ for any $u\in U,\ v\in V$) is  positive definite
                (resp. positive semi-definite). 
               \end{itemize}
				Then we know $T_4-T_3T_1^{-1}T_2\colon V\rw V$ is also  positive definite (resp. positive semi-definite).
			\end{lem}
			\begin{proof}
               We only need to consider the case that $T$ is positive definite.
				Let $S$ be the linear map defined by $\left(\begin{array}{cc}
					I & -T_1^{-1}T_2\\
					0 & I
				\end{array}\right),$ where $I$ is the identity map, then $S$ is invertible. We claim that 
				$$S^*=\left(\begin{array}{cc}
					I & 0\\
					-(T_1^{-1}T_2)^* & I
				\end{array}\right).$$ 
				In fact, for any $u_1,u_2\in U,v_1,v_2\in V$, we have (as $U\bot V$)
				\begin{align*}
                    &\quad \langle u_1+v_1,S^*(u_2+v_2)\rangle=\langle S(u_1+v_1),u_2+v_2\rangle \\
					&=\langle u_1-(T_1^{-1}T_2)v_1+v_1,u_2+v_2\rangle\\
                    &=\langle u_1+v_1,u_2+v_2\rangle-\langle T_1^{-1}T_2v_1,u_2+v_2\rangle \\
					&=\langle u_1+v_1,u_2+v_2\rangle-\langle T_1^{-1}T_2v_1,u_2\rangle\\
                    &=\langle u_1+v_1,u_2+v_2\rangle-\langle v_1,(T_1^{-1}T_2)^*u_2\rangle \\
					&=\langle u_1+v_1,u_2+v_2\rangle-\langle u_1+v_1,(T_1^{-1}T_2)^*u_2\rangle\\
                    &=\langle u_1+v_1,u_2-(T_1^{-1}T_2)^*u_2+v_2\rangle 
				\end{align*}
				Now for any $v\in V\setminus\{0\}$, we have
				\begin{align*}   
					0&<\langle TSv,Sv\rangle=\langle S^*TS v, v\rangle =\langle S^*T(-T_1^{-1}T_2v+v),v\rangle\\
                    &=\langle S^*(-T_1(T_1^{-1}T_2v)+T_2v-T_3(T_1^{-1}T_2v)+T_4v),v\rangle \\
					&=\langle S^*(T_4v-T_3T_1^{-1}T_2v),v\rangle=\langle (T_4-T_3T_1^{-1}T_2)v,Sv\rangle \\
					&=\langle (T_4-T_3T_1^{-1}T_2)v,v\rangle 
				\end{align*}
				 then we know $T_4-T_3T_1^{-1}T_2$ is positive definite.
			\end{proof}
           \begin{cor}\label{cor:positive definite}
            For any $z\in\overline\Omega$, $\left(\left(\rho_{j\bar{k}}-\sum_{p,s}\rho_{j\bar{s}}\rho^{s\bar{p}}\overline{\rho_{k\bar{p}}}\right)(z)\right)_{1\leq j,k\leq n}$ is a positive semi-definite matrix for any $z\in\overline\Omega$.
            If $\Omega$ is a strictly pseudoconvex family of bounded domains over $U$, then $\left(\left(\rho_{j\bar{k}}-\sum_{p,s}\rho_{j\bar{s}}\rho^{s\bar{p}}\overline{\rho_{k\bar{p}}}\right)(z)\right)_{1\leq j,k\leq n}$ is a positive 
            definite matrix for any $z\in\overline\Omega$.
            \end{cor}
          The following Lemma is a direct generalization of Theorem A.5 of \cite{Wang17}.
          \begin{lem}\label{lem:curvature estimate}
          If $F$ is Nakano strictly positive on $\Omega$. Fix $t_0\in U$, and suppose there are an open neighborhood $V\subset U$ of $t_0$,
          $v_1\in C^1(V,F)$ and $v_2\in C^0(V,F)$ such that $v:=\sum_s dz_s\wedge D_{z_s}^Fv_1+v_2$ is $\bar{\partial}_z$-closed,
          then we may solve $\bar{\partial}_zu=v$ on $\Omega_t$ ($t\in V$) with an estimate 
          $$
          \int_{\Omega_t}h^F(u,u)d\lambda\leq \int_{\Omega_t}h^F(v_1,v_1)d\lambda+\int_{\Omega_t}h^F(A^{-1}v_2,v_2)d\lambda,
          $$
          provided the right hand side is finite.
          \end{lem}
          Now we give the proof of Lemma \ref{lem:strictly positive vector bundle}.
	      \begin{lem}[= Lemma \ref{lem:strictly positive vector bundle}]\label{lem:compute}
            If $F$ is  Nakano strictly positive on $\Omega$, then (locally)
            \begin{align*}
             &\quad\sum_{j,k,\lambda,\mu}\int_{\Omega_t}H(h^F)_{jk\lambda\mu}u_{j\lambda}\overline{u_{k\mu}}d\lambda-\sum_{j,k}\int_{\Omega_t}h^F(\pi_{\perp}(L_ju_j),\pi_{\perp}(L_ku_k))d\lambda
             \\
             &\quad+\sum_{j,k,p,q}\int_{\Omega_t}\frac{\partial\nu_j^p}{\partial\bar{z}_q}\frac{\partial\overline{\nu_k^q}}{\partial z_p}h^F(u_j,u_k)d\lambda \\
              &\geq  \int_{\Omega_t} \sum_{j,k,\beta,\gamma}(A_{jk\beta\gamma}-\sum_{p,q,\lambda,\mu}B_{pq\lambda\mu}A_{jq\beta\mu}A_{pk\lambda\gamma})u_{j\beta}\overline{u_{k\gamma}} d\lambda  ,
			\end{align*} 
				where $B$ is uniquely determined by $A$ and satisfies
				$$
                \sum_{q,\beta}A_{pq\lambda\beta}B_{sq\alpha\beta}=\delta_{ps}\delta_{\lambda\alpha}.
                 $$
            \end{lem}
            \begin{proof}
        Put 
       $$v:=\sum_j\pi_\perp\left(L_j u_j\right).$$
We will use Lemma \ref{lem:curvature estimate} to give an estimate of $\int_{\Omega_t}h^F(v,v)d\lambda$. Note that 
    	\begin{align*}
            &\quad\bar\partial_zv\\
            &=\sum_j\bar{\partial}_z(\pi_{\perp}(L_j u_j))=\sum_j\bar{\partial}_z(L_j u_j) \\
            &= \sum_{j,\lambda}\bar{\partial}_z(\frac{\partial u_{j\lambda}}{\partial t_j}dz\otimes e_\lambda+u_{j\lambda}dz\otimes D_{t_j}^Fe_\lambda) \\
            &\quad +\sum_{j,p,\lambda} \bar{\partial}_z\nu_j^p\wedge (\frac{\partial u_{j\lambda}}{\partial z_p}dz\otimes e_\lambda+u_{j\lambda}dz\otimes D_{z_p}^Fe_\lambda) \\
            &\quad +\sum_j \nu_j^p \bar{\partial}_z\left(\frac{\partial u_{j\lambda}}{\partial z_p}dz\otimes e_\lambda+u_{j\lambda}dz\otimes D_{z_p}^Fe_\lambda\right)+\sum_{j,p,\lambda}\bar{\partial}_z\left(\frac{\partial\nu_j^p}{\partial z_p}\right)\wedge u_j \\
            &=\sum_{j,\lambda}\bar{\partial}_z(u_{j\lambda}dz\otimes D_{t_j}^Fe_\lambda)+\sum_{j,p,\lambda} \bar{\partial}_z\nu_j^p\wedge \left(\frac{\partial u_{j\lambda}}{\partial z_p}dz\otimes e_\lambda+u_{j\lambda}dz\otimes D_{z_p}^Fe_\lambda\right) \\
            &\quad +\sum_j \nu_j^p \bar{\partial}_z(u_{j\lambda}dz\otimes D_{z_p}^Fe_\lambda)+\sum_{j,p,\lambda}\bar{\partial}_z\left(\frac{\partial\nu_j^p}{\partial z_p}\right)\wedge u_j \\
            &=(-1)^m\sum_{j,\lambda}u_{j\lambda}dz\wedge \bar{\partial}_z(D_{t_j}^Fe_\lambda) \\
            &\quad +(-1)^m\sum_{j,p,q,\lambda} \frac{\partial\nu_j^p}{\partial\bar{z}_q}\left(\frac{\partial u_{j\lambda}}{\partial z_p}dz\wedge d\bar{z}_q\otimes e_\lambda+u_{j\lambda}dz\wedge d\bar{z}_q\otimes D_{z_p}^Fe_\lambda\right) \\
            &\quad +(-1)^m\sum_j \nu_j^p u_{j\lambda}dz\wedge \bar{\partial}_z(D_{z_p}^Fe_\lambda)+(-1)^m\sum_{j,p,q,\lambda}\frac{\partial^2\nu_j^p}{\partial z_p\partial\bar{z}_q}u_{j\lambda}
            dz\otimes d\bar{z}_q\otimes e_\lambda \\
            &=(-1)^m\sum_{j,q,\lambda,\mu}\left(u_{j\lambda}\frac{\partial\Gamma_{j\lambda}^\mu}{\partial\bar{z}_q}+\sum_p\frac{\partial\nu_j^p}{\partial\bar{z}_q}u_{j\lambda}\Gamma_{p\lambda}^\mu+
            \sum_p\nu_j^pu_{j\lambda}\frac{\partial\Gamma_{p\lambda}^\mu}{\partial\bar{z}_q}\right)dz\wedge d\bar{z}_q\otimes e_\mu \\
            &\quad +(-1)^m\sum_{j,p,q,\mu}\left(\frac{\partial\nu_j^p}{\partial\bar{z}_q}\frac{\partial u_{j\mu}}{\partial z_p}+\frac{\partial^2\nu_j^p}{\partial z_p\partial\bar{z}_q}u_{j\mu}
            \right)dz\wedge d\bar{z}_q\otimes e_\mu \\
            &=(-1)^{m+1}\sum_{j,p,\lambda,\alpha}u_{j\alpha}\left(A_{jp\alpha}^\lambda-\sum_s\frac{\partial\nu_j^s}{\partial\bar{z}_p}\Gamma_{s\alpha}^\lambda+
            \sum_s\nu_j^sA_{sp\alpha}^\lambda\right)dz\wedge d\bar{z}_p\otimes e_\lambda \\
            &\quad +(-1)^m\sum_{j,p,s,\lambda}\left(\frac{\partial\nu_j^s}{\partial\bar{z}_p}\frac{\partial u_{j\lambda}}{\partial z_s}+\frac{\partial^2\nu_j^s}{\partial z_s\partial\bar{z}_p}u_{j\lambda}
            \right)dz\wedge d\bar{z}_p\otimes e_\lambda \\
            &=\sum_s dz_s\wedge D_{z_s}^Fv_1+ v_2, 
    	\end{align*}
        where
        $$
         v_1:=\sum_{j,p,q,\lambda} (-1)^{m+p-1}\frac{\partial \nu_j^p}{\partial\bar{z}_q} u_{j\lambda}dz_1\wedge\cdots\wedge\widehat{dz_p}\wedge\cdots\wedge dz_m\wedge d\bar{z}_q\otimes e_\lambda,
        $$
        $$
        v_2:=(-1)^{m+1}\sum_{j,p,\lambda,\alpha}u_{j\alpha}(A_{jp\alpha}^\lambda+
            \sum_s\nu_j^sA_{sp\alpha}^\lambda)dz\wedge d\bar{z}_p\otimes e_\lambda.
        $$
         By Lemma \ref{lem:curvature estimate}, we know 
        $$
          \int_{\Omega_t}h^F(v,v)d\lambda\leq \int_{\Omega_t}h^F(v_1,v_1)d\lambda+\int_{\Omega_t}h^F(A^{-1}v_2,v_2)d\lambda.
          $$
       To prove Lemma \ref{lem:compute},  we only need to verify 
       \begin{align*}
       &\quad h^F( A^{-1}v_2,v_2)\\
        &=\sum_{j,k,p,q,\lambda,\mu,\alpha,\beta}B_{pq\lambda\mu} A_{jp\alpha\mu}A_{pk\lambda\beta}u_{j\alpha}\overline{u_{k\beta}}+\sum_{j,k,p,\alpha,\beta}\overline{\nu_k^p}A_{jp\alpha\beta}u_{j\alpha}\overline{u_{k\beta}} \\
        &\quad +\sum_{j,k,p,\alpha,\beta}\nu_j^p(A_{pk\alpha\beta}+\sum_a\overline{\nu_k^a}A_{pa\alpha\beta})u_{j\alpha}\overline{u_{k\beta}} 
        \end{align*}
        for some $B$ to be determined later.\\
        \indent For any $g:=\sum_{p,\lambda}g_{p\lambda}dz\wedge d\bar{z}_p\otimes e_\lambda,$ we have 
    	\begin{align*}
    	Ag&=i\Theta^{(F,h^F)}(\Lambda_\omega g) \\
    		&=(i\sum_{p,q,\lambda,\mu}A_{pq\lambda}^\mu dz_p\wedge d\bar{z}_q\otimes e_\lambda^*\otimes e_\mu)(\Lambda_\omega g) \\
    		&=\sum_{j,p,q,\lambda,\mu}A_{pq\lambda}^\mu g_{p\lambda}dz\wedge d\bar{z}_q\otimes e_\mu. 
    	\end{align*}
    	Note that we may view $A$ as a linear map $T$ which satisfies 
    	$$
       T(d\bar{z}_p\otimes e_\lambda)= \sum_{q,\mu}A_{pq\lambda}^\mu d\bar{z}_q\otimes e_\mu.
       $$
    	Since $(F,h^F)$ is Nakano strictly positive, then  for any nonzero matrices $(u_{p\lambda})$, we have 
    	$$
       \sum_{p,q,\lambda,\mu}A_{pq\lambda\mu}u_{p\lambda}\overline{u_{q\mu}}>0,
       $$
    	and then $T$ is invertible. Let $T^{-1}$ be the inverse map of $T$, and write 
    	$$
        T^{-1}(d\bar{z}_q\otimes e_\mu)=\sum_{p,\lambda}B_{pq\lambda}^\mu d\bar{z}_p\otimes e_\lambda,
        $$
    	then we have 
    	$$
         d\bar{z}_p\otimes e_\lambda=T^{-1}(\sum_{q,\mu}A_{pq\lambda}^\mu d\bar{z}_q\otimes e_\mu)=\sum_{q,s,\alpha,\mu}A_{pq\lambda}^\mu B_{sq\alpha}^\mu d\bar{z}_s\otimes e_\alpha,
         $$
    	i.e.
        $$
         \sum_{q,\mu}A_{pq\lambda}^\mu B_{sq\alpha}^\mu=\delta_{ps}\delta_{\lambda\alpha}.
         $$
    	Let 
    	$$
         B_{sq\alpha\beta}:=\sum_{\mu} B_{sq\alpha}^\mu (h^F)^{\beta\mu},
         $$
        then we have 
    	$$
        \sum_{q,\beta}A_{pq\lambda\beta}B_{sq\alpha\beta}=\sum_{q,\mu}A_{pq\lambda}^\mu B_{sq\alpha}^\mu=\delta_{ps}\delta_{\lambda\alpha}.
        $$
    	Define linear maps (for the simplicity of notations, we omit $dz$ here)
    	$$
        S_1(d\bar{z}_p\otimes e_\lambda)=\sum_{q,\mu}A_{pq\lambda\mu}d\bar{z}_q\otimes e_\mu,\ S_2(d\bar{t}_j\otimes e_\lambda)=\sum_{q,\mu}A_{jq\lambda\mu}d\bar{z}_q\otimes e_\mu,
        $$
    	$$
       S_3(d\bar{z}_p\otimes e_\lambda)=\sum_{k,\mu}A_{pk\lambda\mu}d\bar{t}_k\otimes e_\mu,\ S_4(d\bar{t}_j\otimes e_\lambda)=\sum_{k,\mu}A_{jk\lambda\mu}d\bar{t}_k\otimes e_\mu,
       $$
    	$$
      B(d\bar{z}_q\otimes e_\mu)=\sum_{p,\lambda}B_{pq\lambda\mu}d\bar{z}_p\otimes e_\lambda,
      $$
    	then we have 
    	$$
       B S_1(d\bar{z}_p\otimes e_\lambda)=B(\sum_{q,\mu}A_{pq\lambda\mu}d\bar{z}_q\otimes e_\mu)=
    	\sum_{q,s,\mu,\alpha}A_{pq\lambda\mu}B_{sq\alpha\mu}d\bar{z}_s\otimes e_\alpha=d\bar{z}_p\otimes e_\lambda,
        $$
    	so $B=S_1^{-1}$. \\
    	\indent Let 
         $$
         (v_2)_{p\lambda}:=(-1)^{m+1}\sum_{j,\alpha}u_{j\alpha}(A_{jp\alpha}^\lambda+
            \sum_s\nu_j^sA_{sp\alpha}^\lambda),
         $$
         then we know  
    	$$
        A^{-1}v_2=\sum_{p,q,\lambda,\mu} B_{pq\lambda}^\mu v_{q\mu}dz\wedge d\bar{z}_p\otimes e_\lambda,
       $$
    	and 
    	\begin{align*}
        &\quad h^F( A^{-1}v_2,v_2)\\
        &=\sum_{j,k,p,q,\lambda,\mu,\alpha,\beta,\gamma}B_{pq\lambda}^\mu (A_{jp\alpha}^\mu+\sum_s \nu_j^s A_{sq\alpha}^\mu)\overline{(A_{kp\beta}^\gamma+\sum_a\nu_k^a
         A_{ap\beta}^\gamma)}u_{j\alpha}\overline{u_{k\beta}}h_{\lambda\gamma}^F \\
        &=\sum_{j,k,p,q,\lambda,\mu,\alpha,\beta}B_{pq\lambda}^\mu (A_{jp\alpha}^\mu+\sum_s \nu_j^s A_{sq\alpha}^\mu)(A_{pk\lambda\beta}+\sum_a\overline{\nu_k^a}A_{pa\lambda\beta})u_{j\alpha}\overline{u_{k\beta}} \\
        &=\sum_{j,k,p,\lambda,\alpha,\beta}(\sum_{q,\mu}B_{pq\lambda}^\mu A_{jp\alpha}^\mu+ \nu_j^p\delta_{\alpha\lambda})(A_{pk\lambda\beta}+\sum_a\overline{\nu_k^a}A_{pa\lambda\beta})u_{j\alpha}\overline{u_{k\beta}} \\
        &=\sum_{j,k,p,q,\lambda,\mu,\alpha,\beta}B_{pq\lambda}^\mu A_{jp\alpha}^\mu(A_{pk\lambda\beta}+\sum_a\overline{\nu_k^a}A_{pa\lambda\beta})u_{j\alpha}\overline{u_{k\beta}} \\
        &\quad +\sum_{j,k,p,\alpha,\beta}\nu_j^p(A_{pk\alpha\beta}+\sum_a\overline{\nu_k^a}A_{pa\alpha\beta})u_{j\alpha}\overline{u_{k\beta}} \\
        &=\sum_{j,k,p,q,\lambda,\mu,\alpha,\beta}B_{pq\lambda\mu} A_{jp\alpha\mu}(A_{pk\lambda\beta}+\sum_a\overline{\nu_k^a}A_{pa\lambda\beta})u_{j\alpha}\overline{u_{k\beta}} \\
        &\quad +\sum_{j,k,p,\alpha,\beta}\nu_j^p(A_{pk\alpha\beta}+\sum_a\overline{\nu_k^a}A_{pa\alpha\beta})u_{j\alpha}\overline{u_{k\beta}} \\
        &=\sum_{j,k,p,q,\lambda,\mu,\alpha,\beta}B_{pq\lambda\mu} A_{jp\alpha\mu}A_{pk\lambda\beta}u_{j\alpha}\overline{u_{k\beta}}+\sum_{j,k,p,\alpha,\beta}\overline{\nu_k^p}A_{jp\alpha\beta}u_{j\alpha}\overline{u_{k\beta}} \\
        &\quad +\sum_{j,k,p,\alpha,\beta}\nu_j^p(A_{pk\alpha\beta}+\sum_a\overline{\nu_k^a}A_{pa\alpha\beta})u_{j\alpha}\overline{u_{k\beta}} .
        \end{align*}
        Thus the proof is completed.
    	\end{proof}
		\end{document}